\documentclass[english, 11pt]{amsart}
\usepackage{euscript,graphicx,epstopdf,amscd,amsgen,amsfonts,amssymb,latexsym,amsmath,amsthm,graphicx,mathrsfs,times,color,overpic,fourier,array}
\usepackage[all]{xy}
\usepackage{enumitem}
 \usepackage[colorlinks=true, backref=page]{hyperref}
\usepackage{xcolor}
\definecolor{forestgreen}{rgb}{0.0, 0.27, 0.13}

\usepackage[T1]{fontenc}

\newtheorem{teo}{Theorem}[section]

\newtheorem{lemma}[teo]{Lemma}
\newtheorem{claim}[teo]{Claim}

\newtheorem*{claim*}{Claim}
\newtheorem{coro}[teo]{Corollary}
\newtheorem{prop}[teo]{Proposition}

\newtheorem{mtheo}{Theorem}

\theoremstyle{definition}
\newtheorem{defi}[teo]{Definition}
\newtheorem{notation}[teo]{Notation}
\newtheorem{fact}[teo]{Fact}

\newtheorem{remark}[teo]{Remark}
\newtheorem{af}[teo]{Claim}
\newtheorem*{maintheorem*}{Theorem}

\DeclareMathOperator{\diam}{diam}

\DeclareMathOperator{\Diff}{Diff}

\def\c{{\rm c}}

\def\s{{\rm s}}

\def\u{{\rm u}}
\def\cs{{\rm cs}}
\def\cu{{\rm cu}}

\def\bN{\mathbb{N}}

\def\bZ{\mathbb{Z}}

\def\bR{\mathbb{R}}

\def\c{{\rm c}}
\def\s{{\rm ss}}
\def\u{{\rm uu}}
\def\cs{{\rm cs}}
\def\cu{{\rm cu}}

 \def\N{\mathbb{N}}

\def\N{\mathbb{N}}

\def\Ga{\Gamma}

\def\le{\chi}

\newcommand{\eqdef}{\stackrel{\scriptscriptstyle\mathrm{def}}{=}}

\numberwithin{equation}{section}

\DeclareMathSymbol{\varnothing}{\mathord}{AMSb}{"3F}
\renewcommand{\emptyset}{\varnothing}

\title[Zero entropy aperiodic ergodic measures]{Denseness of zero entropy aperiodic ergodic measures}

\author[C.~Crispim]{Camila Crispim}
\email{crispimscamila@gmail.com}

\author[L.~J.~D\'iaz]{Lorenzo J. D\'\i az}

\email{lodiaz@puc-rio.br}

\address{Departamento de Matem\'atica PUC-Rio, Marqu\^es de S\~ao Vicente 225, G\'avea, Rio de Janeiro 22451-900, Brazil}

\thanks{This paper is part of the thesis of the first author, completed under the supervision of the second author at PUC-Rio. 
It has been supported [in part] by CAPES - Finance Code 001. 
CC was supported by a PhD fellowship from CAPES. 
LD was supported by CNPq grants 310069/2020-3 and Projeto Universal
401737/2025-0 and  404943/2023-3, as well as by E-16/2014 INCT/FAPERJ and E-26/204.046/2024 CNE/FAPERJ. 
\newline
The authors thank D. Kwietniak for his comments on GIKN and rank-one measures, and CC thanks the colleagues and staff at PUC-Rio for their support and valuable discussions}

\begin{document}

\begin{abstract}
We study partially hyperbolic homoclinic classes of $C^1$-generic diffeomorphisms with a one-dimensional central bundle, so that the central Lyapunov exponent $\chi^c(\mu)$ is well defined for any ergodic measure $\mu$ supported on the class. We focus on nonhyperbolic homoclinic classes supporting ergodic measures with positive, zero, and negative central exponents. For each $\alpha$ and a nontrivial homoclinic class $H$ of a $C^1$-generic diffeomorphism $f$, we consider the level set of measures 
\[
\mathcal{M}^\alpha_{\mathrm{erg}}(f,H)=
\left\{\text{$\mu$ ergodic, supported on $H$, with } \chi^c(\mu)=\alpha \right\}.
\]  
In this generic setting, the range of $\alpha$ for which $\mathcal{M}^\alpha_{\mathrm{erg}}(f,H)$ is nonempty forms a nontrivial closed interval $I$. Since the set of periodic measures is countable, most of these sets contain no periodic measures. We show that for every $\alpha$ in the interior of $I$, the so-called Axiomatized GIKN measures, a class of low-complexity, zero-entropy measures, are dense in $\mathcal{M}^\alpha_{\mathrm{erg}}(f,H)$. 
This result can be viewed as an analogue of Sigmund's classical density of periodic measures for systems with the specification property, obtained here in a setting where the specification property does not hold and periodic measures are typically absent (in the considered level sets).

We also present a similar result for the open class of blender-minimal diffeomorphisms, contained in the class of 
$C^1$-robustly transitive ones.
\end{abstract}


\maketitle



\section{Introduction}
In 1974, Sigmund~\cite{Sig:74} proved that the specification property implies the density (in the weak$^\ast$ topology) of \emph{periodic measures} (that is, ergodic measures supported on periodic orbits) in the space of invariant measures. This result can be regarded as an ergodic counterpart of the density of periodic orbits in that setting. The specification property, introduced by Bowen in~\cite{Bow:72}, roughly asserts that finite orbit segments can be freely concatenated and shadowed by true orbits. 

In the differentiable setting, hyperbolicity and the specification property are closely related, with the former implying the latter. In general, however, the above density result fails without the specification property, as there may exist convex combinations of ergodic measures that are not limits of ergodic measures\footnote{There are many examples of this phenomenon; see, for instance,~\cite{DiaHorRioSam:09}.}. Under the specification property, Sigmund~\cite{Sig:70} also provides a detailed description of the space of invariant measures.

There are extensions of Sigmund's theorem beyond the specification property in the setting of $C^1$-generic diffeomorphisms (i.e., diffeomorphisms in a residual subset of $\mathrm{Diff}^1(M)$, the space of $C^1$-diffeomorphisms). In this direction, a main result asserts that every invariant measure supported on an isolated transitive set is the weak$^\ast$ limit of periodic measures supported on the same set; see \cite[Theorem~3.5(a)]{AbdBonCro:11}. This substantially strengthens a classical consequence of Ma\~n\'e's ergodic closing lemma~\cite{Man:82}, which states that, for $C^1$-generic diffeomorphisms, invariant measures are weak$^\ast$ limits of convex combinations of periodic orbit measures.

Our goal is to formulate (semi-local and global) relative versions of Sigmund's results for systems without specification, by considering ``level sets of measures'' in partially hyperbolic settings. We focus on \emph{transitive} systems (i.e., systems with dense orbits) or on isolated transitive sets admitting a dominated invariant splitting into three nontrivial bundles, two of which are uniformly hyperbolic (contracting and expanding), while the third is a one-dimensional central bundle. Throughout the paper, whenever we refer to partial hyperbolicity, we assume the existence of such a three-bundle splitting, and we refer to the corresponding sets or splittings as \emph{strongly partially hyperbolic}.

In this setting, the central bundle is part of the Oseledets splitting associated to any ergodic measure $\mu$, which allows us to define its \emph{central Lyapunov exponent} $\chi^c(\mu)$. Ergodic measures are thus classified according to the sign of this exponent. Those with zero central Lyapunov exponent are called \emph{nonhyperbolic} and are the main focus of this paper.  We stress that nonhyperbolicity of the system does not imply the existence of nonhyperbolic measures. Indeed, there are nonhyperbolic systems whose ergodic measures are all hyperbolic; see \cite{BalBonSch:99,DiaHorRioSam:09}. Such situations are typically regarded as pathological.

Given a map $f$, denote by $\mathcal{M}_{\mathrm{erg}}(f)$ the set of its ergodic measures, and by $\mathcal{M}_{\mathrm{erg}}(f,\Lambda)$ those supported on a closed $f$-invariant set $\Lambda$. Let $f$ be a diffeomorphism and $\Lambda$ a closed, $f$-invariant, transitive, partially hyperbolic set. For $\alpha \in \mathbb{R}$, define
\begin{equation}
\label{e.measurelevelsets}
\mathcal{M}^\alpha_{\mathrm{erg}}(f,\Lambda)
\eqdef
\left\{
\mu \in \mathcal{M}_{\mathrm{erg}}(f,\Lambda)
\colon
\chi^c(\mu)=\alpha
\right\}.
\end{equation}
For many values of $\alpha$ this set is empty, but in relevant cases the set of $\alpha$ for which $\mathcal{M}^\alpha_{\mathrm{erg}}(f,\Lambda)\neq\emptyset$ is a closed interval. Our main interest is the nonhyperbolic case $\alpha=0$, which is the most delicate; the case $\alpha\neq 0$ is analogous and simpler.

A question is whether a relative version of Sigmund's result holds for~\eqref{e.measurelevelsets}, namely whether periodic measures are dense in $\mathcal{M}^\alpha_{\mathrm{erg}}(f,\Lambda)$. This fails in general: by the Kupka--Smale genericity theorem, periodic points are generically hyperbolic and hence countable, so most of these sets contain no periodic measures.

Furthermore, since entropy is upper semicontinuous in our setting~\cite{DiaFis:11} and these sets typically contain measures with positive entropy, any dense subclass must consist of zero-entropy measures\footnote{We are 
assuming that entropy is one of the ingredients describing the class.}. Thus, to obtain a density result in the spirit of Sigmund's theorem, periodicity must be replaced by a notion of low dynamical complexity, namely zero entropy. Note that nonhyperbolic measures with positive entropy do exist; see~\cite{BocBonDia:16}.

Thus, we must discard the periodic case and look for an appropriate replacement for periodic measures. Recently, a class of low-complexity measures has been explored in several contexts: the \emph{GIKN measures} (or, more precisely, \emph{axiomatized GIKN measures}), introduced in \cite{Goretal:05} and defined in \cite{KwiLac:} following the axiomatization in \cite{BonDiaGor:10}. 
These measures are obtained via periodic approximation: they arise as ergodic weak$^\ast$ limits of sequences of periodic measures in which each orbit shadows the previous one for most of the time, while also exhibiting a \emph{tail} that deviates from it. A precise quantification of these shadowing and tailing times is required. We present and discuss these measures in Section~\ref{s.GIKN}. GIKN measures are always ergodic \cite{Goretal:05,BonDiaGor:10}, have zero entropy \cite{KwiLac:}, 
and belong to the class of \emph{rank-one} measures, that is, measures that can be approximated using a single Rokhlin tower \cite{KwiLacTri:}. In particular, they can be regarded as the simplest type of aperiodic ergodic measures. For a general discussion of rank-one measures and their relevance, see \cite{Fer:97}.

We present two settings in which aperiodic GIKN-measures are dense in the sets
$\mathcal{M}^\alpha_{\mathrm{erg}}(f,\Lambda)$.
First, in the setting of partially hyperbolic $C^1$-generic diffeomorphisms, in Theorem~\ref{t.teogeralclassehom}
we consider the case where $\Lambda$ is an isolated homoclinic class. 
We observe that certain genericity assumptions are necessary, since otherwise one may find so-called exposed pieces of the dynamics, where the range of Lyapunov exponents exhibits gaps and the results may fail; see \cite{DiaHorRioSam:09} and related variations in \cite{DiaGelRam:14,Diaetal:19}.

The second setting (see Theorem~\ref{t.densitygeraldifeo})
considers the case $\Lambda = M$ within the class of $C^1$-robustly transitive diffeomorphisms introduced in \cite{DiaGelSan:20,YanZha:20}, known as blender-minimal. This class encompasses many known examples of robustly nonhyperbolic transitive sets. We provide a precise definition of this class in Section~\ref{ss.setmb}.

\begin{notation}
Throughout the whole paper, $f \in \mathrm{Diff}^1(M)$, where  $M$ is a compact boundaryless manifold.
We endow $\mathrm{Diff}^1(M)$ with the uniform topology.
We denote by  $\mathcal{M}_{\mathrm{inv}} (f)$ the set of $f$-invariant measures. When considering measures supported on an $f$-invariant closed set $\Lambda$, we use the notations 
 $\mathcal{M}_{\mathrm{inv}} (f, \Lambda)$. The notations
 $\mathcal{M}_{\mathrm{erg}} (f)$ and 
  $\mathcal{M}_{\mathrm{erg}} (f, \Lambda)$ were introduced above. 
\end{notation}

\subsection{Nonhyperbolic homoclinic classes}
We now introduce the class $\mathrm{ISPH}(f)$ of homoclinic classes of a diffeomorphism $f$ that will be a key object of our study. To this end, we first recall some terminology.

A {\em saddle} $p$ is a hyperbolic periodic point whose stable and unstable bundles are both nontrivial. The {\em$\mathrm{s}$-index} of a saddle is the dimension of its stable bundle. Recall that a saddle $p$ has well-defined (saddle) continuations $p_g$ for every diffeomorphism $g$ sufficiently $C^1$-close to $f$.

The homoclinic class of a saddle $p$, denoted by $H(p)$, is the closure of the transverse intersections between the stable and unstable manifolds of its orbit. See Section~\ref{s.homrelations} for details and \cite[Chapter 10.4]{BonDiaVia:05} for a discussion of the relevance of homoclinic classes.

An $f$-invariant set $\Lambda$ is {\em isolated for $f$} if there exists a neighborhood $U$ of $\Lambda$, called an {\em isolating block of $\Lambda$}, such that
\begin{equation} \label{e.LambdaU}
\Lambda_f(\overline{U}) = \Lambda_f(U)=\Lambda,
\quad \text{where} \quad
\Lambda_f(K) \eqdef \bigcap_{i \in \bZ} f^i(K).
\end{equation}

A homoclinic class $H(p)$ is said to have \emph{index variation} if it contains saddles with different $\mathrm{s}$-indices. This property implies that the class is nonhyperbolic. We say that $H(p)$ is \emph{$C^1$-robustly isolated} if there exists a $C^1$-neighborhood $\mathcal{U}$ of $f$ such that the homoclinic class $H(p_g)$ is isolated for every $g \in \mathcal{U}$.

\begin{defi}[Homoclinic classes in $\mathrm{ISPH}(f)$]
\label{d.ISPH}
Given $f \in \mathrm{Diff}^1(M)$, let $\mathrm{ISPH}(f)$ denote the collection of homoclinic classes $H$ of $f$ such that:
\begin{enumerate}
\item $H$ is strongly partially hyperbolic;
\item $H$ has index variation; and
\item $H$ is $C^1$-robustly isolated.
\end{enumerate}
\end{defi}

Associated to a class $H \in \mathrm{ISPH}(f)$, and recalling the definition of $\mathcal{M}_{\mathrm{erg}}^\alpha(f,H)$ in \eqref{e.measurelevelsets}, we define
\begin{equation}
\label{eq.defalphaminalphamax}
\alpha_{\inf}(H) \eqdef \inf \left\{ \alpha \colon \mathcal{M}_{\mathrm{erg}}^\alpha(f,H) \neq \emptyset \right\},
\,\,
\alpha_{\sup}(H) \eqdef \sup \left\{ \alpha \colon \mathcal{M}_{\mathrm{erg}}^\alpha(f,H) \neq \emptyset \right\}.
\end{equation}

Finally, consider the subset $\mathcal{M}_{\mathrm{GIKN}}(f)$ of $\mathcal{M}_{\mathrm{erg}}(f)$ consisting of GIKN measures, and its subset $\mathcal{M}_{\mathrm{aGIKN}}(f)$ formed by aperiodic ones (see Definitions~\ref{d.GIKNseq} 
and \ref{d.GIKNmes} 
and Theorem~\ref{t.l.p.strongnonper}).
Similarly, for $\alpha \in \mathbb{R}$, we define the corresponding subsets $\mathcal{M}^\alpha_{\mathrm{GIKN}}(f,\Lambda)$ and $\mathcal{M}^\alpha_{\mathrm{aGIKN}}(f,\Lambda)$ as in \eqref{e.measurelevelsets}, replacing ergodic measures by GIKN and aGIKN measures, respectively.

\begin{mtheo}
\label{t.teogeralclassehom}
There exists a residual set $\mathcal{R}(M)$ of $\mathrm{Diff}^1(M)$ such that for every $f \in \mathcal{R}(M)$ and every homoclinic class $H(p_f)\in \mathrm{ISPH}(f)$, there exists a neighborhood $\mathcal{U}_f$ of $f$ such that
$$
\mathcal{M}^\alpha_{\mathrm{aGIKN}}(g,H(p_g)) \text{ is dense in } \mathcal{M}^\alpha_{\mathrm{erg}}(g,H(p_g))
$$
for every $g \in \mathcal{R}(M) \cap \mathcal{U}_f$ and every
$\alpha \in \big(\alpha_{\inf}(H(p_g)),\alpha_{\sup}(H(p_g))\big)$.
\end{mtheo}

The most interesting case of Theorem~\ref{t.teogeralclassehom} corresponds to the ``nonhyperbolic case,'' when $\alpha = 0$. This case is proved in Section~\ref{s.thmAnonhyperbolic}. The ``hyperbolic cases'', that is, $\alpha \neq 0$, follow by suitable modifications of the argument for $\alpha = 0$ and are treated in Section~\ref{s.hyperbolic}. For these hyperbolic values, the index variation hypothesis is not required; see Theorem~\ref{t.hyperboliccase}.

\subsection{Blender-minimal diffeomorphisms}
\label{ss.setmb}

To define the (open) subset $\mathbf{MB}^1(M)$ of $\mathrm{Diff}^1(M)$ consisting of blender-minimal diffeomorphisms, we first introduce some preliminaries.

We denote by $\mathbf{PH}_{\mathrm{c}=1}^1(M)$ the $C^1$-open subset of $\mathrm{Diff}^1(M)$ formed by diffeomorphisms for which the whole ambient manifold is a strongly partially hyperbolic set. For each $f \in \mathbf{PH}_{\mathrm{c}=1}^1(M)$, there are well-defined {\em strong stable} $\mathcal{F}^\s$ and {\em strong unstable} $\mathcal{F}^\u$ foliations, tangent to the corresponding invariant bundles; see Section~\ref{s.ph} for details.

Given a foliation $\mathcal{F}$ of $M$ and $x\in M$, we denote by $\mathcal{F}(x)$ the leaf containing $x$. A foliation $\mathcal{F}$ is said to be {\em minimal} if $\mathcal{F}(x)$ is dense in $M$ for every $x\in M$. Note that minimality of either of the strong foliations of $f \in \mathbf{PH}_{\mathrm{c}=1}^1(M)$ implies that $f$ is transitive.

A final ingredient in the definition of $\mathbf{MB}^1(M)$ is given by {\em{stable and unstable blender-horseshoes,}} introduced in \cite{BonDia:12}. For a detailed discussion, see Section~\ref{ss.blenders}, and for an informal presentation, see \cite{Bonetal:16}. Roughly speaking, a blender-horseshoe is a higher-dimensional horseshoe (of dimension at least three) satisfying a superposition property, stated in Lemma~\ref{l.sobrediscos}. A key feature is that blender-horseshoes are ubiquitous in the robustly nonhyperbolic and robustly transitive setting; see \cite{BonDia:08}.

We are now ready to define the set $\mathbf{MB}^1(M)$, following \cite{DiaGelSan:20,YanZha:20}.

\begin{defi}
\label{d.setmb}
The set $\mathbf{MB}^1(M)$ is the subset of $\mathbf{PH}_{\mathrm{c}=1}^1(M)$ consisting of diffeomorphisms such that
\begin{enumerate}[leftmargin=2cm,label=(\textbf{MB\arabic*}),start=1]
    \item \label{MB1} both the strong stable and strong unstable foliations are minimal; and
    \item \label{MB2} $f$ has both a stable blender-horseshoe and an unstable blender-horseshoe.
\end{enumerate}
\end{defi}
 
 At first glance, the defining assumptions of $\mathbf{MB}^1(M)$ may seem restrictive; yet this class is remarkably large and rich in structure. Indeed, conditions~\ref{MB1} and~\ref{MB2} together guarantee that every diffeomorphism in $\mathbf{MB}^1(M)$ is transitive and nonhyperbolic. While the existence of a blender-horseshoe is $C^1$-open~\cite{BonDia:12}, minimality of a strong foliation is not open in general. The presence of blender-horseshoes ensures the openness of this property as well~\cite[Theorem~4]{BonDiaUre:02}, making all diffeomorphisms in $\mathbf{MB}^1(M)$ robustly transitive. Moreover, by~\cite{BocBonDia:16}, for each $f \in \mathbf{MB}^1(M)$, the set $\mathcal{M}_{\mathrm{erg}}^0(f)$ is nonempty and contains measures with positive topological entropy.

The relevance of $\mathbf{MB}^1(M)$ has been widely discussed; see \cite[Section~2.4]{DiaGelRam:} for a detailed overview. In dimension three, restricting to robustly transitive diffeomorphisms, this class is $C^1$-dense among diffeomorphisms fibered by circles~\cite{BonDia:96,BonDiaUre:02,RodRodUre:07} and among flow-type diffeomorphisms~\cite{BonDia:96,BuzFisTah:22}. It also contains open subsets of derived-from-Anosov diffeomorphisms~\cite{RodUreYan:22} as well as certain anomalous examples~\cite{BonGogHamPot:20}.

As in~\eqref{eq.defalphaminalphamax}, for $f \in \mathrm{Diff}^1(M)$, we define
\begin{equation} \label{e.maxmin}
    \alpha_{\inf}(f) \eqdef \inf \{ \alpha \colon \mathcal{M}_{\mathrm{erg}}^\alpha(f) \neq \emptyset \}, \qquad
    \alpha_{\sup}(f) \eqdef \sup \{ \alpha \colon \mathcal{M}_{\mathrm{erg}}^\alpha(f) \neq \emptyset \}.
\end{equation}

\begin{mtheo}
\label{t.densitygeraldifeo}
For every $f \in \mathbf{MB}^1(M)$ and every $\alpha \in (\alpha_{\inf}(f),\alpha_{\sup}(f))$, 
the set $\mathcal{M}^\alpha_{\mathrm{aGIKN}}(f)$ is dense in $\mathcal{M}^\alpha_{\mathrm{erg}}(f)$. 
\end{mtheo}

This paper is organized as follows. Section~\ref{ss.framework} introduces the main concepts and tools used throughout: GIKN measures, homoclinic classes and their $C^1$-generic properties, partial hyperbolicity and strong foliations, ergodic measures and Lyapunov exponents, blender-horseshoes, and split flip-flop configurations. Section~\ref{s.distancetoGIKN} presents a setting guaranteeing the density of aperiodic nonhyperbolic GIKN measures (Theorem~\ref{t.densityisolatedset}). Section~\ref{s.splitflipflopinsidehomoclinicclass} discusses split flip-flop configurations in homoclinic classes with index variation (Theorem~\ref{t.classeblender}). The proof of Theorem~\ref{t.teogeralclassehom} is divided into nonhyperbolic measures (Section~\ref{s.thmAnonhyperbolic}) and hyperbolic measures (Section~\ref{s.hyperbolic}). Finally, Section~\ref{s.proofthmB} contains the proof of Theorem~\ref{t.densitygeraldifeo}.

\begin{notation}[Convergence of measures]
Throughout the whole paper, convergence of a sequence of measures means convergence in the
 weak$^\ast$ topology.
\end{notation}

\section{Framework, main definitions, and tools}
\label{ss.framework}

We now introduce the main ingredients of our constructions. In Section~\ref{s.GIKN}, we define GIKN measures and state their properties. Section~\ref{s.homrelations} discusses properties of homoclinic classes for $C^1$-generic diffeomorphisms. In Section~\ref{s.ph}, we recall the notion of partial hyperbolicity together with the associated strong stable and unstable foliations. Section~\ref{s.ergodicmeasures} introduces hyperbolic and nonhyperbolic ergodic measures. In Section~\ref{ss.blenders}, we present blender-horseshoes and the properties relevant to our purposes. Finally, Section~\ref{s.flipflops} introduces split flip-flop configurations

\subsection{Axiomatized GIKN measures}
\label{s.GIKN}
Given a point $x \in M$, 
its {\em{forward orbit}} is the set 
    $
        \mathcal{O}(x) \eqdef \{ f^n(x) \colon n \geq 0\}.
    $
The orbit $\mathcal{O}(x)$ and $x$ are {\em{periodic}} if there exists $\pi\in \N$, such that $f^\pi(x)=x$. The smallest number $\pi(x)=\pi (\mathcal{O}(x))$ satisfying that condition is the {\em{period}} of $\mathcal{O}(x)$ or $x$.

An invariant measure  of $f$ is {\em{periodic}} if it is ergodic and supported on a periodic orbit, otherwise is {\em{aperiodic}}. 
Given a periodic orbit $\mathcal{O}(x)$,
we denote by $\mu_{\mathcal{O}(x)}$ the unique ergodic measure supported on $\mathcal{O}(x)$.

\begin{defi}[$(\epsilon,\pi)$-good points]
\label{d.goodpoints}
    Let $\mathcal{O}(x)=\mathcal{O}$ be a periodic orbit of $f$ and $\epsilon>0$. A point $y \in M$ is $(\epsilon,\pi(x))$-{\em{good}} for 
    $\mathcal{O}$ if there exists $\bar x \in \mathcal{O}$ such that 
    $$
        d(f^k(\bar x),f^k(y))< \epsilon \quad \mbox{for every } k \in \{0,1,\dots,\pi(x)-1\}.
    $$

    Given a  periodic orbit $\mathcal{O}(y)$,
    we denote by  $\mathcal{O}(y)_{\epsilon}^{\mathrm{g}}(\mathcal{O})$ the subset of
    $\mathcal{O}(y)$ 
      of points which  are $\epsilon$-good for $\mathcal{O}$. 
     We let 
     $$
     \mathcal{O}(y)_{\epsilon}^{\mathrm{b}}(\mathcal{O} ) \eqdef \mathcal{O}(y) \backslash \mathcal{O}(y)_{\epsilon}^{\mathrm{g}}(\mathcal{O}).
     $$ 
     We call the  sets   $\mathcal{O}(y)_{\epsilon}^{\mathrm{g}}(\mathcal{O})$
     and  $\mathcal{O}(y)_{\epsilon}^{\mathrm{b}}(\mathcal{O})$
   the
    $\epsilon$-{\em{good}} and $\epsilon$-{\em{bad points of $\mathcal{O}(y)$ with respect to $\mathcal{O}$}}, respectively.
\end{defi}

In what follows, given a finite set $K$, we denote by $\# (K)$ its cardinality.

\begin{defi}[$(\epsilon,\rho)$-good approximations]
\label{d.goodap}
    Let $\mathcal{O}(x)=\mathcal{O}$ be a periodic orbit and $\epsilon>0$ such that the balls $\{B_{\epsilon}(f^i(x))\}_{i=0}^{\pi(x)-1}$ 
    are pairwise disjoint. Given $\rho \in (0,1)$, a periodic orbit $\mathcal{O}(y)$ is an $(\epsilon,\rho)$-{\em{good approximation}} of $\mathcal{O}$ if
    \begin{enumerate}
        \item $\dfrac{\# \big(\mathcal{O}(y)_{\epsilon}^{\mathrm{g}}(\mathcal{O}) \big)}{\#\big(\mathcal{O}(y) \big)} \geq \rho$;
        \item there is $\kappa$ such that $\#\big( \mathcal{O}(y)_{\epsilon}^{\mathrm{g}}(\mathcal{O} )
         \cap B_{\epsilon}(f^i(x)) \big)= \kappa$ for all $i \in \{0,\dots,\pi(x)-1\}$.
    \end{enumerate}
\end{defi}

\begin{defi}[GIKN sequences]
\label{d.GIKNseq}
   A sequence of periodic orbits  $(\mathcal{O}_n)_{n}$ of a diffeomorphism $f$ with increasing periods  is called 
   {\em{GIKN}} if there are sequences of positive numbers $(\epsilon_n)_n$ and $(\rho_n)_n$ with 
        $$
        \sum_{n = 1}^{\infty} \epsilon_n < \infty \quad \mbox{and} \quad \prod_{n=1}^{\infty} \rho_n \in (0,1]
        $$
    such that $\mathcal{O}_{n+1}$ is an $(\epsilon_n,\rho_n)$-good approximation of $\mathcal{O}_n$ for every $n$. 
    
    In this case, we say that the sequences of 
    orbits $\mathcal{O}_n$ and
    measures  $(\mu_{\mathcal{O}_n})_n$ are
    {\em{$(\epsilon_n,\rho_n)$-GIKN}} 
    or simply {\em{GIKN.}}
\end{defi}

In the theorem below, convergence follows from~\cite[Lemma 2.5]{BonDiaGor:10} and~\cite[Lemma 2]{Goretal:05}, aperiodicity from~\cite[Section 8]{Goretal:05}, the entropy statement from~\cite{KwiLac:}
and the rank-one assertion from \cite{KwiLacTri:}.

\begin{teo}
\label{t.l.p.strongnonper}
    Let $(\mu_n)_n$ be a $(\epsilon_n, \rho_n)$-GIKN sequence of measures 
    supported on orbits $\mathcal{O}_n$
    such that for every $n$
    it holds
        \begin{equation}
        \label{e.distorbitaseq}
            \sum_{k=n}^{\infty} \epsilon_k < \frac{d_n}{3}, \qquad 
            \mbox{where} \quad  d_n \eqdef d_{\mathcal{O}_n}  =\inf \{d(y,z) \colon y \ne z, \, \, y,z \in \mathcal{O}_n\}.
        \end{equation}
 Then $(\mu_n)_n$ converges 
 to an ergodic aperiodic  rank-one  measure with zero topological entropy. 
\end{teo}

\begin{defi}[Axiomatized GIKN measures]
\label{d.GIKNmes}
        A measure that is the limit of a GIKN sequence is called {\em{Axiomatized GIKN}} or simply {\em{GIKN.}}
\end{defi}

\subsection{Homoclinic classes of $C^1$-generic diffeomorphisms}
\label{s.homrelations}

We now introduce homoclinic classes and state their properties for $C^1$-generic diffeomorphisms.

Let  $p$ be  a saddle of $f$.  
The {\em{stable}} and {\em{unstable}} sets of $\mathcal{O}(p)$ are, respectively, defined by
 \[
 \begin{split}
 &W^\mathrm{s}(\mathcal{O}(p)) \eqdef \{ x \colon \lim_{n \to \infty} d(f^n(x),\mathcal{O}(p))=0\},
 \\
& W^\mathrm{u}(\mathcal{O}(p)) \eqdef \{ x \colon \lim_{n \to \infty} d(f^{-n}(x),\mathcal{O}(p))=0\}.
 \end{split}
 \]
 For  $\ast \in \{ \mathrm{s},\mathrm{u}\}$,
let $W^\ast (p)$ the connected component of $W^\ast (\mathcal{O}(p))$ that contains $p$.

We say that two periodic orbits (or points) are {\em{homoclinically related}} if 
    $$
        W^\mathrm{s}(\mathcal{O}(p)) \pitchfork  W^\mathrm{u}(\mathcal{O}(q)) \neq \emptyset \neq W^\mathrm{u}(\mathcal{O}(p)) \pitchfork  W^\mathrm{s}(\mathcal{O}(q)).
    $$
Every point  $x \in W^{\mathrm{s}}(\mathcal{O}(p)) \pitchfork W^{\mathrm{u}}(\mathcal{O}(p))$, $x\not\in \mathcal{O}(p)$, is a called a {\em{ transverse homoclinic point}} of $\mathcal{O}(p)$.
The {\em{homoclinic class}} of $\mathcal{O}(p)$ is defined by
 \begin{equation}
 \label{eq.defhomclass}
 H(p) \eqdef \overline{\{ W^\mathrm{s}(\mathcal{O}(p)) \pitchfork W^\mathrm{u}(\mathcal{O}(p))\}}.
 \end{equation}
A homoclinic class is a {\em{transitive set}} (it is the closure of some  orbit of it) whose periodic points form a  dense subset of it, see for instance~\cite[Chapter 5]{Wen:16}.

A homoclinic class is {\em{nontrivial}} if it contains at least two different orbits.

We now introduce the chain recurrent set of $f$.
Given $\epsilon>0$, a sequence $(x_i)_{i \in \mathbb{Z}}$ is an $\epsilon$-pseudo orbit of $f$ if $d(f(x_i),x_{i+1})<\epsilon$, for every $i \in \mathbb{Z}$.  We consider also finite pseudo orbits. The $\epsilon$-pseudo orbit $(x_i)_{i\in \mathbb{Z}}$ is {\em{periodic}} if the sequence $(x_i)_{i\in \mathbb{Z}}$ is periodic.

The {\em{chain recurrent set of $f$,}} 
denoted by $\mathrm{CR}(f)$ is the set of  points $x \in M$ such that for every $\epsilon>0$, there exists  a
finite $\epsilon$-pseudo orbit starting and ending at $x$. 
Two points $x,y \in \mathrm{CR}(f)$ are in the same {\em{chain recurrence class}} if for every $\epsilon>0$, there is an $\epsilon$-pseudo orbit starting at $x$ and ending at $y$ and vice versa. The chain recurrence class of $x$ is denoted by $\mathrm{CR}(x)$. Two recurrence classes are equal or disjoint and are $f$-invariant closed sets, see~\cite{Wen:16}. 

\begin{remark}
\label{r.classerecorrenciatransitivo}
Every transitive set is contained in some chain recurrence class. In particular,
$H(p) \subset \mathrm{CR}(p)$, for every hyperbolic periodic point $p$ of $f$.
\end{remark}

Recall that the {{$\mathrm{s}$-index}} of a hyperbolic periodic point $p$, denoted by $\mathrm{s}$-$\mathrm{ind}(p)$, is the dimension of its stable bundle. 

A periodic point $p$ of $f$ has {\em{real multipliers}} if every eigenvalue of $Df^{\pi(p)}(p)$ is real and has multiplicity one, and every two eigenvalues of $Df^{\pi(p)}(p)$ have different absolute value. 
We denote by $\mathrm{Per}_{\mathbb{R}}(H(p))$ the set of saddles in $H(p)$ having real multipliers.

For the result below, recall that every saddle $p$ of $f$ has a continuation $p_g$ for every diffeomorphism
$g$ sufficiently $C^1$-close to $f$. This saddle has the same $\mathrm{s}$-index as $p$.

\begin{teo}[see \cite{CarMorPac:03,Abdetal:07,BonCro:04,Cheetal:19}]
\label{t.residual}
There is a residual subset $\mathcal{R}_0(M)$ of $\mathrm{Diff}^1(M)$ 
consisting of diffeomorphisms $f$ satisfying conditions (G0)-(G6) below: 
\begin{enumerate}[leftmargin=2cm,label=(\textbf{G\arabic*}),start=0]
    \item\label{G0} The periodic points of $f$ are hyperbolic and their invariant manifolds intersect transversally; 
    \item\label{G1}
     $\mathrm{CR}(p)=H(p)$ for every periodic point $p$ of $f$. In particular, every pair of homoclinic classes of $f$ are either equal or disjoint;
    \item \label{G2} 
    every pair of saddles of $f$ with the same $\mathrm{s}$-index in the same  homoclinic class are homoclinically related;
    \item \label{G3} 
    every homoclinic class depends continuously on $f$ (in the Hausdorff metric);
    \item \label{G4} 
    for every pair of saddles $p$ and $q$ of $f$, there is a neighborhood $\mathcal{U}$ of $f$ such that for every $g \in \mathcal{U}\cap \mathcal{R}_0(M)$
    either  $H(p_g)=H(q_g)$ 
    or  $H(p_g)\cap H(q_g)=\emptyset$;
    \item \label{G5} for every saddle $p$ of $f$, the set $\mathrm{Per}_{\mathbb{R}}(H(p))$ is dense in $H(p)$;
    and
    \item  \label{G6}
  every nonhyperbolic homoclinic class of $f$ supports a nonhyperbolic measure.
\end{enumerate}
\end{teo}

\subsection{Partial Hyperbolicity and strong foliations}
\label{s.ph}

An $f$-invariant compact set $\Lambda \subset M$ is {\em{strongly partially hyperbolic}} if there are
a $Df$-invariant  splitting with three nontrivial bundles
\begin{equation}
\label{e.threesplitting}
T_{\Lambda}M = E^{\s} \oplus E^{\mathrm{c}} \oplus E^{\u},
\qquad \mbox{where} \quad \dim (E^\c)=1,
\end{equation}
and constants $C>1$ and $\lambda\in(0,1)$ such that for every $x\in \Lambda$ and $n\in\mathbb{N}$, it holds
\begin{itemize}
	\item $\|Df^n|_{E^{\s}(x)}\|<C \lambda^n$ and $\|Df^{-n}|_{E^{\u}(x)}\|<C\lambda^n$;
	\item $\max\left\{ \|Df^n|_{E^{\s}(x)}\|\, \|Df^{-n}|_{E^\c(f^n(x))}\| \,,\, \|Df^n|_{E^{\c}(x)}\|\, \|Df^{-n}|_{E^{\u}(f^n(x))}\|\right\}<C\lambda^n$.
	\end{itemize}
The bundles of the splitting depend continuously on $x$, see~\cite[Section B.1]{BonDiaVia:05}.
 By \cite{Gou:07}, 
 there is a metric, equivalent to the initial one, called {\em{adapted,}} where $C=1$. In what follows, we will consider
 adapted metrics.

The bundles $E^{\s}$, $E^{\u}$, and $E^{\mathrm{c}}$ are called {\em{strong stable}}, {\em{strong unstable}} and {\em{center bundles}}, respectively. When $\Lambda=M$, we say that the diffeomorphism $f$ is {\em{strongly partially hyperbolic.}}

We now consider the  {\em{strong stable}} and {\em{unstable laminations}} $\mathcal{F}^\s$ and $\mathcal{F}^\u$ tangent to $E^\s$ and $E^\u$, respectively.   When $\Lambda = M$, the laminations $\mathcal{F}^\s$ and $\mathcal{F}^\u$ are, indeed, foliations.
The following is a classical result that we formulate for $\mathcal{F}^\s$, writing $\mathcal{F}^\s(x)$ for the leaf through $x$. The analogous statement for $\mathcal{F}^\u$ is obtained by considering $f^{-1}$.

\begin{teo}[\cite{HirPugShu:77}]
\label{t.invariantfoliation}
    Let $f \in \Diff^1(M)$ and $\Lambda$ be a strongly partially hyperbolic set of $f$. There is a uniquely defined $f$-invariant 
    lamination $\mathcal{F}^{\s}$  such that every point of $x\in\Lambda$ 
    is contained in some leaf and
    \begin{enumerate}
        \item the leaves of $\mathcal{F}^{\s}$  are immersed submanifolds of $M$ of dimensions $\mathrm{dim}(E^{\s})$
        such that 
        $ T_{x}\mathcal{F}^{\s}=E^{\s}(x)$ for every $x\in \Lambda$;
        \item there is $\rho>0$ such that 
          $$
                y \in \mathcal{F}^\s(x) \Leftrightarrow \lim_{n \to \infty} \frac{1}{n} \log{d(f^n(x),f^n(y))}<-\rho;
           $$
            \item the leaves depend continuously on $x$:  if $x_n \to x$, then $\mathcal{F}^\s(x_n) \to \mathcal{F}^\s(x)$.
    \end{enumerate}
\end{teo}

\begin{remark}
    Let $\Lambda \subset M$ be a strongly partially hyperbolic set  of $f$   and $p\in \Lambda$ be a saddle. 
    Then   either $\mathrm{s}$-$\mathrm{ind}(p)=\dim(E^\s)$ or $\mathrm{s}$-$\mathrm{ind}(p)=\dim(E^\s)+1$.
    In the first case, 
    $\mathcal{F}^\u(p) \subsetneq  W^{\mathrm{u}}(p)$ and $\mathcal{F}^\s(p) = W^{\mathrm{s}}(p)$.
    In the second case, 
    $\mathcal{F}^\u(p) = W^{\mathrm{u}}(p)$ and $\mathcal{F}^\s(p) \subsetneq W^{\mathrm{s}}(p)$.   
\end{remark}

\subsection{Ergodic measures and center Lyapunov exponents}
\label{s.ergodicmeasures}

Let $\mu\in\mathcal{M}_{\mathrm{erg}} (f)$.
The Oseledets Multiplicative Ergodic Theorem provides a set $\Upsilon$ with $\mu (\Upsilon)=1$, a $Df$-invariant 
splitting $T_\Upsilon M=E_1 \oplus \dots \oplus E_m$,
called {\em{Oseledets splitting of $\mu$,}}
 and numbers $\chi_1(\mu) < \dots < \chi_m(\mu)$, called {\em{Lyapunov exponent of $\mu$}}, such that 
    $$
        \chi_i(\mu) \eqdef \lim_{n \to \pm \infty} \frac{1}{n} \log \| Df^n|_{E_i(x)} \|, \qquad \mbox{for every $y \in \Upsilon$ and $i\in\{1,\dots,m\}$.}
    $$
An ergodic measure is {\em{hyperbolic}} if $\chi_i (\mu) \ne 0$
for every $i$, otherwise it is
called {\em{nonhyperbolic.}}

For what is below,
 we consider a strongly partially hyperbolic set $\Lambda$ of $f$  as in \eqref{e.threesplitting}
 and the set $\mathcal{M}_{\mathrm{erg}} (f,\Lambda)$.
For  $\mu \in \mathcal{M}_{\mathrm{erg}} (f,\Lambda)$,
there are three possibilities for a bundle $E_i$ of the Oseledets splitting of $\mu$:
$E_i\subset E^\s$,  $E_i \subset E^\u$ or $E_i=E^{\mathrm{c}}$. 
Since the strong stable and unstable bundles are uniformly contracting and expanding, respectively, the Lyapunov exponent associated to $E_i$  is negative if $E_i \subset E^\s$ or  positive if $E_i \subset E^\u$.
Concerning the center bundle $E^\c$, we have $E^\c =E_j$ for some bundle of the Oseledets splitting, thus
the exponent $\chi^{\mathrm{c}} (\mu)=\chi_j (\mu)$, which is called the {\em{center Lyapunov exponent of $\mu$,}} such that
    $$
        \chi^{\mathrm{c}} (\mu) \eqdef \lim_{n \to \pm \infty} \frac{1}{n} \log \| Df^n|_{E^{\mathrm{c}}(x)} \|,
        \qquad \mbox{for $\mu$-almost every point $x$.}
    $$

As a consequence,  $\mu \in \mathcal{M}_{\mathrm{erg}} (f,\Lambda)$
is nonhyperbolic if and only if $\chi^{\mathrm{c}}(\mu)=0$, otherwise it is hyperbolic.
Note that 
\begin{equation}
\label{e.threemeasures}
        \mathcal{M}_{\mathrm{erg}}(f, \Lambda) = \mathcal{M}^+_{\mathrm{erg}}(f, \Lambda) \cup \mathcal{M}^0_{\mathrm{erg}}(f, \Lambda) \cup \mathcal{M}_{\mathrm{erg}}^-(f, \Lambda),
   \end{equation}
where $\mathcal{M}^0_{\mathrm{erg}}(f, \Lambda)$ is the set of  nonhyperbolic measures of $\Lambda$, and $\mathcal{M}^+_{\mathrm{erg}}(f, \Lambda)$ and $\mathcal{M}^-_{\mathrm{erg}}(f, \Lambda)$ are the sets of hyperbolic measures of $\Lambda$ with positive and negative center Lyapunov exponents, respectively.      

    \begin{remark}
    \label{r.convexp}
    Consider the potential $\varphi^{\mathrm{c}} \colon M \to \mathbb{R}$ defined by 
        \begin{equation}
        \label{eq.derivadacentrala}
            \varphi^{\mathrm{c}}(x) \eqdef \log \| Df|_{E^{\mathrm{c}}(x)} \|.
        \end{equation}
       By the comments above, this map is continuous. Thus
    given any $\mu \in \mathcal{M}_{\mathrm{erg}} (f)$, by the Birkhoff ergodic theorem
    and since $E^\c$ is one-dimensional,
        \begin{equation}
        \label{eq.explyapint}
            \chi^{\mathrm{c}}(\mu) = \lim_{n \to \infty} \frac{1}{n}\sum_{i=0}^n \varphi^{\mathrm{c}}(f^n(x)) = \int \varphi^{\mathrm{c}}  d \mu.
        \end{equation}
    Hence, if a sequence of ergodic measures $(\mu_n)_n$ 
    converges to an ergodic measure $\mu$ then the sequence $(\chi^{\mathrm{c}}(\mu_n))_n$ converges to $\chi^{\mathrm{c}}(\mu)$.
    \end{remark}

\subsection{Blender-horseshoes}
\label{ss.blenders}
An \emph{unstable blender-horseshoe} $B^+=(f, \mathbf{C}, \Gamma)$ is a triple consisting of a diffeomorphism $f$, a compact set $\mathbf{C}$, and a transitive hyperbolic set $\Gamma \subset \mathrm{int}(\mathbf{C})$, which is the maximal invariant set of $f$ in $\mathbf{C}$. This set is strongly partially hyperbolic, and its central bundle is contained in the unstable one. With a slight abuse of notation, we also refer to $\Gamma$ as a blender-horseshoe.
The complete definition involves conditions (BH1)-(BH6) in \cite[Section 3.2]{BonDia:12}; see also \cite[Section 2]{DiaGelSan:20} for further discussion and properties. We will not give the full definition here, focusing instead on the key property needed for our purposes, stated in Lemma~\ref{l.sobrediscos}.

A \emph{stable blender-horseshoe} is defined as an unstable blender-horseshoe for the inverse diffeomorphism $f^{-1}$."

\begin{remark}[Blender-horseshoes have continuations]
\label{r.blendercontinuations}
As a hyperbolic set, every blender-horseshoe has a hyperbolic continuation which also is  
 a blender-horseshoe, see \cite[Lemma 3.9]{BonDia:12}.
 More precisely, if $B^+=(f, \mathbf{C}, \Ga)$ is a blender-horseshoe then
 $B^+_g=(g, \mathbf{C}, \Ga_g)$, where $\Ga_g$ is the continuation of $\Ga$,
 is also a blender-horseshoe for every $g$ close enough to $f$.
\end{remark} 
 
Given a splitting $E=E_1 \oplus E_2$ and a constant $\alpha \in (0,1)$ the {\em{cone field 
around $E_1$ of size $\alpha$,}} is the subset of $E$
defined by
 $$
 \mathcal{C}_\alpha (E_1) \eqdef \{ v= v_1+v_2 \colon v_i \in E_i , || v_2|| \leq \alpha || v_1||\}.
 $$
 A cone field on a set $K\subset M$ is a continuous map  defined on $K$, $x\mapsto C_{\alpha(x)} (E_1(x))$, where
 the dimensions of the spaces $E_1(x)$ are constant. This cone field is $Df$-invariant if
 $Df(C_{\alpha(x)} (E_1(x))$ is strictly contained in $C_{\alpha(f(x))} (E_1(f(x)))$, for every $x\in K \cap f^{-1}(K)$.
 Note this $Df$-invariant cone field can be extended to a $Df$-invariant cone field in a neighborhood of $K$.
 Having a $Df$-invariant cone field on a compact set is a $C^1$-open property. For details, see
 \cite[Appendix B1]{BonDiaVia:05}.
 
 In what follows, we consider  a strongly partially hyperbolic set $\Lambda$
 with a splitting 
$T_{\Lambda}M = E^{\s} \oplus E^{\mathrm{c}} \oplus E^{\u}$ as in \eqref{e.threesplitting},
  and
 cone fields  $\mathcal{C}^\u$ and $\mathcal{C}^{\mathrm{cu}}$,  
  $\mathcal{C}^\u\subset \mathcal{C}^{\mathrm{cu}}$,
 around the
 bundles $E^\u$ and $E^\c \oplus E^\u$, respectively. We will omit the size of these cones.
 The existence of these cone fields follows from the partial hyperbolicity.
 We can freely  assume that the splittings and the cone fields are extended to a neighborhood of $\Lambda$.

\begin{remark}[Cone fields and central expansion constant] 
\label{r.cuexpansion}
A blender-horseshoe
 $B^+$ 
 is also endowed with a pair of $Df$-invariant cone fields 
 $\mathcal{C}^\u\subset \mathcal{C}^{\mathrm{cu}}$ as above
defined on $\mathbf{C}$, see condition (BH2) in \cite[Section 3.2]{BonDia:12}.
 As the central bundle $E^\c\oplus E^\u$ of a unstable blender-horseshoe $\Ga$ is uniformly expanded, 
    there is a   {\em{central expansion constant}}
    $\lambda_{\mathrm{bh}}(B^+)>1$
       such that
            $ \| Df_x(v)\| \geq \lambda_{\mathrm{bh}} \|v\|$
    for every $v\in \mathcal{C}^\mathrm{\cu}$.
 \end{remark}

For our goals, the main property of  blender-horseshoes is the following:

\begin{remark}[The family of disks in-between]
\label{r.defdisksinbet}
    Conditions (BH4)-(BH6) in~\cite{BonDia:12} 
    provides 
    an open  family   
     $\mathfrak{D}_{\mathrm{bet}}(B^+)$
    of $f$-invariant disks of dimension $\dim (E^\u)$ tangent to  $\mathcal{C}^\u$, called
   disks {\em{in-between,}} 
    such that the image $f(D)$ of any $D\in \mathfrak{D}_{\mathrm{bet}}(B^+)$ contains a disk
    in $\mathfrak{D}_{\mathrm{bet}}(B^+)$.
\end{remark}

The next lemma is an immediate consequence of the definition of $\mathfrak{D}_{\mathrm{bet}}(B^+)$:
\begin{lemma} 
\label{l.sobrediscos}
Let $B^+$ a unstable blender-horseshoe $f$. Then     
   every $D \in \mathfrak{D}_{\mathrm{bet}}(B^+)$ and every $N \in \N$, there exists $D_N \subset D$ such that
   \begin{enumerate}
       \item $f^i(D_N)$ is contained in some disk in $\mathfrak{D}_{\mathrm{bet}}(B^+)$ for all $i \in \{0,\dots,N\}$, 
       and
       \item $f^N(D_N)$ contains a disk in $\mathfrak{D}_{\mathrm{bet}}(B^+)$.   
       \end{enumerate}
\end{lemma}

The next remark is a direct consequence of the previous lemma:

\begin{remark}\label{r.intersectionwiththestableset}
Consider a unstable blender-horseshoe  $B^+=(f, \mathbf{C}, \Ga)$ 
and the local stable manifold $\Ga$ given by
$$
W^{\mathrm{s}}_{\mathrm{loc}} (\Ga) \eqdef  \bigcap_{i\in \mathbb{N}} f^i (\mathbf{C}).
$$
Then every disk  $D \in \mathfrak{D}_{\mathrm{bet}}(B^+)$ intersects $W^{\mathrm{s}}_{\mathrm{loc}} (\Ga)$,
 see \cite[Remark 3.10]{BonDia:12}.
\end{remark}

 \subsection{Split flip-flop configurations}
\label{s.flipflops}
Following ~\cite{BocBonDia:16}, 
we proceed to introduce {\em{split flip-flop configurations}} associated to 
a unstable blender-horseshoe $B^+$
and  saddle $p$ with $\chi^\c (p)<0$.
This definition involves
the  family $\mathfrak{D}_{\mathrm{bet}}(B^+)$ of disks in-between of $B^+$ and 
a family $\mathfrak{D}_{\delta, \rho}(p)$ extending a local unstable manifold of $p$
that we proceed to describe, see Remark~\ref{rd.familiadiscos}.

Given  small $r>0$, we denote 
by $W^{\mathrm{s}}_r(p)$ and $W^{\mathrm{u}}_r(p)$
the local stable  and unstable sets of $p$ of size $r$, respectively, and for $\ast \in \{\s,\u\}$, we define
    $
        \mathcal{F}^\ast_r(x)
    $
as the connected component  of $\mathcal{F}^\ast(x) \cap B_r(x)$ that contains $x$.

\begin{remark}[Safety region]
\label{r.safetyregionofsaddles}
There is $\delta_1(p)>0$ such that  
    \begin{equation}
        f^{\pi(p)}(W^\mathrm{s}_{\delta}(p)) \subset W^\mathrm{s}_{\delta(p)}(p)
        \quad \mbox{for every $\delta \in (0, \delta_1(p))$.}
    \end{equation}
Consider the continuous map 
\begin{equation}
\label{e.varphipip}
\varphi_p^\c \colon \Lambda \to \bR, \qquad
\varphi_p^\c (x) \eqdef \log {\|Df^{\pi(p)}|_{E^{\mathrm{c}}(x)}\|}.
\end{equation}
Since $\varphi^\c_p(p)<0$, there is $\delta_2(p)>0$ such that  
$\varphi^\c_p(x)<0$ for every $x$ with $d(x,p)<\delta_2(p)$.
We set $\delta(p) \eqdef \min \{\delta_1(p),\delta_2(p)\}$
and call it a {\em{safe number of $p$.}}
\end{remark}

Denote by $\mathcal{D}^\u(M)$ the set of all closed disks with dimension $\dim (E^\u)$ that are $C^1$-embedded in $M$. By~\cite[Proposition 3.1]{BocBonDia:16}, the space $\mathcal{D}^\u (M)$ is metrizable with a distance $\vartheta$.
Given a 
a family of discs 
$\mathfrak{D}$ in  $\mathcal{D}^\u(M)$,
let $\mathcal{V}_\epsilon (\mathfrak{D})$ 
the $\epsilon$-neighborhood of $\mathfrak{D}$ in the metric $\vartheta$.
The family $\mathfrak{D}$ is
 {\em{strictly $f$-invariant}}
if there is $\epsilon>0$ such that the image $f(D)$ of any disk 
$D\in \mathcal{V}_\epsilon (\mathfrak{D})$  contains a disk in $\mathfrak{D}$.

\begin{remark}[Family $\mathfrak{D}_{\delta,\rho}(p)$]
\label{rd.familiadiscos}
 Due to hyperbolicity of $p$, there are small $\rho,\delta>0$
such that the family $\mathfrak{D}_{\delta,\rho}\eqdef \mathcal{V}_\delta ( W^{\mathrm{u}}_\rho (p))$
is strictly invariant for $f^{\pi (p)}$. Moreover, we can assume that:
\begin{enumerate}
\item
every disk in $\mathfrak{D}_{\delta,\rho}$ is contained in $B_{\delta(p)} (p)$,
\item
every disk in $\mathfrak{D}_{\delta,\rho}$ intersects transversely $W^{\mathrm{s}}_{\delta(p)/4}(p)$,
and
\item
there is a $Df^{\pi (p)}$-invariant unstable cone field $\mathcal{C}^{\u}$
such that every disk in   $D\in \mathfrak{D}_{\delta,\rho}$ is tangent to $\mathcal{C}^\u$
(i.e. $T_x D \subset \mathcal{C}^\u(x)$ for every $x\in D$).
\end{enumerate}
Note that the disks in  $\mathfrak{D}_{\delta,\rho}$ contains a ball of uniform size.
For details, see \cite[Lemma 4.11]{BocBonDia:16}.
\end{remark}

\begin{defi}[Split flip-flop configuration]
\label{d.splitflipflop}
A periodic orbit $\mathcal{O}(p)$ with $\chi^{\mathrm{c}}(p)<0$ and an unstable blender-horseshoe $B^+$ form a {\em{split flip-flop configuration}} if
\begin{enumerate}
        \item if $W^\mathrm{u}(\mathcal{O}(p))$ contains a disk in $\mathfrak{D}_\mathrm{bet}(B^+)$;
        \item  there are $\delta,\rho>0$ and $N \eqdef N_{\delta,\rho}>0$ 
        such that   $f^{N}(D)$  contains a disk $D' \in \mathfrak{D}_{\delta,\rho}(p)$ for every $D \in \mathfrak{D}_{\mathrm{bet}}(B^+)$.
    \end{enumerate}
\end{defi}
Analogously and taking $f^{-1}$, we define split flip-flop configurations associated to saddles with positive central exponent and stable blender-horseshoes.

 The next remark is a direct consequence from the previous definition, transversality, and the $\lambda$-lemma.
 
\begin{remark}
\label{r.connectiontime}
    Item (1) in Definition~\ref{d.splitflipflop} implies that there is $C \eqdef C_{\delta,\rho}>0$ such that 
for every $D \in \mathfrak{D}_{\delta,\rho}(p)$,    
    $f^{C}(D)$ contains a disk  in $\mathfrak{D}_{\mathrm{bet}}(B^+)$. To emphasize the role of the family of disks 
  $\mathfrak{D}_{\delta,\rho}(p)$  
    and the  numbers $N_{\delta,\rho}$ and $C_{\delta,\rho}$, called {\em{connecting times}} between $\mathfrak{D}_{\delta,\rho}(p)$ and $\mathfrak{D}_{\mathrm{bet}}(B^+)$, and between  $\mathfrak{D}_{\mathrm{bet}}(B^+)$ and $\mathfrak{D}_{\delta,\rho}(p)$, respectively,
    we use the following notation to denote a split flip-flop configuration between $\mathcal{O}(p)$ and $B^+$:      
     \begin{equation}
    \label{e.splitfilpflop}
    \mathrm{SFF}(\mathfrak{D}_{\delta,\rho}(p),\mathfrak{D}_{\mathrm{bet}}(B^+),N_{\delta,\rho},C_{\delta, \rho}).   
     \end{equation}
\end{remark}

The following  is an immediate consequence of the definitions of a  flip-flop configuration:

\begin{remark}
\label{r.refinedflipfop}
Let $ \mathrm{SFF}(\mathfrak{D}_{\delta,\rho}(p),\mathfrak{D}_{\mathrm{bet}}(B^+),N_{\delta,\rho},
C_{\delta,\rho})$ be a flip-flop configuration. Given $\delta_0 \in (0,\delta), \rho_0\in (0, \rho)$.
there are 
$
N_{\delta_0, \rho_0}\ge N_{\delta, \rho,p}$ and $C_{\delta_0,\rho_0}\ge  C_{\delta,\rho,}$ 
such that
 $ \mathrm{SFF}(\mathfrak{D}_{\delta_0, \rho_0}(p),\mathfrak{D}_{\mathrm{bet}}(B^+),N_{\delta_0, \rho_0},C_{\delta_0, \rho_0})$
is a flip-flop configuration.
\end{remark}

We finish with an immediate consequence of the $\lambda$-lemma whose proof is omitted.

\begin{lemma}
\label{l.flipflopeverypoint}
    Let $f \in \mathrm{Diff}^1(M)$, 
    $p$ a saddle of $f$ with $\le^\c (p)<0$  such that
    $H(p)$ is strongly partially hyperbolic,  and $B^+$ an unstable blender-horseshoe.
   Suppose that $p$ and $B^+$ forms a split flip-flop configuration. Then $B^{+}$ forms a split flip-flop configuration with any saddle in $H(p)$ which is homoclinically related to $p$.
\end{lemma}

Note that, the hypotheses of Lemma~\ref{l.flipflopeverypoint} implies that
there is a strongly partially hyperbolic set containing $H(p) \cup B^+$.

\section{Density of nonhyperbolic aperiodic GIKN measures}
\label{s.distancetoGIKN}

The main result of this section is Theorem~\ref{t.densityisolatedset}, which gives sufficient conditions for the density of aperiodic nonhyperbolic GIKN measures among nonhyperbolic measures
supported on a set $\Lambda$.
Its core ingredient is Proposition~\ref{p.mainprop}, which associates to each hyperbolic periodic measure~$\mu
\in \mathcal{M}_{\rm erg} (f, \Lambda)$
 a measure  in 
 $\mathcal{M}_{\rm aGIKN} (f, \Lambda)$
  whose distance to~$\mu$ is bounded by
$
A\,\bigl|\chi^{\mathrm{c}}(\mu)\bigr|,
$
for some constant~$A$ depending only on $\Lambda$.

Given two probability measures $\mu$ and $\nu$ their  {\em{Wasserstein distance}} is
\begin{equation}
\label{eq.wasserstein}
W_1(\mu,\nu)\eqdef  \sup_{\psi \in \mathrm{Lip}(1)}\left\{\left| \int_M \psi(x) d \mu- \int_M \psi(x) d \nu \right|\right\},
\end{equation}
where
$$
\mathrm{Lip}(1)\eqdef \{ \psi \colon M \to \mathbb{R} \colon  \psi \,\, \mbox{is Lipschitz with constant 
 $L \leq 1$}\}.
$$
This distance induces the weak$^\ast$ topology on the space of probability measures (see \cite[Theorem 6.9]{Vil:09}).

\begin{remark}
Fix $x_0 \in M$. It is enough to consider the supremum over the Lipschitz maps $\psi \in \mathrm{Lip}(1)$, with $\psi(x_0)=0$. With that we get that for every $x \in M$, 
$$|\psi(x)|=|\psi(x)-\psi(x_0)|\leq d(x,x_0) \leq \mathrm{diam}(M).$$
\end{remark}

Let $\Lambda$ be a strongly partially hyperbolic set.
Given  $\alpha \in \bR$, let 
    \begin{align*}
        & \mathcal{M}^\alpha_{\mathrm{erg}}(f,\Lambda)  \eqdef \{ \mu \in \mathcal{M}_{\mathrm{erg}}(f,\Lambda) \colon \chi^\c(\mu)=\alpha\} \quad \mbox{and} \\
        & \mathcal{M}^\alpha_{\mathrm{aGIKN}}(f,\Lambda)  \eqdef \mathcal{M}_{\mathrm{aGIKN}}(f) \cap \mathcal{M}^\alpha_{\mathrm{erg}}(f,\Lambda).
    \end{align*}

\begin{teo}
\label{t.densityisolatedset}
    Let $f \in \mathrm{Diff}^1(M)$ and $\Lambda$ be an isolated strongly partially hyperbolic set of $f$
    such that:
        \begin{enumerate}
            \item 
            the set $\Lambda$ contains an unstable blender-horseshoe $B^+$  which
            forms a split flip-flop configuration with any saddle $p \in \Lambda$ with $\chi^\c(p)<0$; and
            \item every measure in $\mathcal{M}_{\mathrm{erg}}^0(f,\Lambda)$ is 
            weak$^\ast$ approximated by periodic
        measures with negative central Lyapunov exponent.
        \end{enumerate}
   Then  
        the set $\mathcal{M}^0_{\mathrm{aGIKN}}(f,\Lambda)$ is dense in $\mathcal{M}^0_{\mathrm{erg}}(f,\Lambda)$.
\end{teo}

The main step to prove this theorem is the following proposition:

\begin{prop}
\label{p.mainprop}
    Consider $f \in \mathrm{Diff}^1(M)$ and $\Lambda$ satisfying the hypotheses of
    Theorem~\ref{t.densityisolatedset}.
   Then  there is a constant $A_{\Lambda}>0$ such that for every periodic measure
   $\mu\in \mathcal{M}_{\mathrm{erg}}(f,\Lambda)$ with $\chi^\c(\mu)<0$ 
   there exists  $\nu \in \mathcal{M}^0_{\mathrm{aGIKN}}(f,\Lambda)$ such that 
        $$
            W_1(\mu,\nu) < A_{\Lambda}|\chi^\c(\mu)|.
        $$
\end{prop}

There are the corresponding versions of Theorem~\ref{t.densityisolatedset} and
 Proposition~\ref{p.mainprop} for stable blender-horseshoes and saddles with positive central Lyapunov exponent.

This section is organized as follows. In Section~\ref{ss.proofoft.densityisolatedset}, we derive Theorem~\ref{t.densityisolatedset} as a direct consequence of Proposition~\ref{p.mainprop}. In Section~\ref{ss.proofofp.mainprop}, we state Lemma~\ref{l.periodicpointshomoclinicclass} and use it to prove Proposition~\ref{p.mainprop}; the proof of the lemma is postponed to Section~\ref{ss.proofofperiodicpointshomoclinicclass}.

\subsection{Proof of Theorem~\ref{t.densityisolatedset}}
\label{ss.proofoft.densityisolatedset}
   Given any $\eta \in \mathcal{M}^0_{\mathrm{erg}}(f,\Lambda)$, we construct a sequence $(\nu_n)_n$ of measures in $\mathcal{M}^0_{\mathrm{aGIKN}}(f,\Lambda)$ which  converges to $\eta$.
    By hypotheses, there exists a sequence  of periodic measures $(\mu_n)_n$ with 
    $\chi^\c(\mu_n)<0$ converging to $\eta$. By Proposition~\ref{p.mainprop}, there exists a constant $A_{\Lambda}>0$ such that for
    every $n \in \bN$, there exists $\nu_{n}\in \mathcal{M}^0_{\mathrm{aGIKN}}(f,\Lambda)$ such that 
        $$
            W_1(\mu_n,\nu_n) < A_\Lambda | \chi^\c(\mu_n)|.
        $$
        Therefore
        \begin{equation*}
            W_1(\eta,\nu_n) \leq W_1(\eta,\mu_n) + W_1(\mu_n,\nu_n) < W_1(\eta,\mu_n) + A_{\Lambda} | \chi^\c(\mu_n)|.
        \end{equation*}
    Since $W_1(\eta,\mu_n) \to 0$ and by Remark~\ref{r.convexp} $\chi^\c(\mu_n) \to \chi^\c(\nu)=0$, we get       
     $$
            \lim_{n \to \infty} W_1(\eta,\nu_n)=0,
       $$
proving the theorem. \qed

\subsection{Proof of Proposition~\ref{p.mainprop}}
\label{ss.proofofp.mainprop}    
    For $\Lambda$ as in the proposition,  let 
    \begin{equation}
    \label{e.Lc}
    L^\c(\Lambda) \eqdef \sup_{x \in \Lambda} |\varphi^\c(x)|, \qquad 
    \mbox{where $\varphi^\c \colon \Lambda \to \bR$ is as in
    \eqref{eq.derivadacentrala}.}
    \end{equation}    
    We use the next lemma, whose proof is postponed to 
    Section~\ref{ss.proofofperiodicpointshomoclinicclass}. In this lemma 
    $\lambda_{\mathrm{bh}}(B^+)>1$ is the expansion constant of the blender-horseshoe in Remark~\ref{r.cuexpansion}.

    \begin{lemma}
    \label{l.periodicpointshomoclinicclass} 
        There is $\epsilon_0>0$
         such that for every 
        $\epsilon \in (0,\epsilon_0)$ and every saddle $p \in \Lambda$ with $\chi^{\mathrm{c}}(p)<0$ there is a saddle $q \in \Lambda$ such that 
        \begin{itemize}
            \item $\pi(q)>\pi(p)$;
            \item $\chi^{\mathrm{c}}(q)<0$ and 
                $
                    |\chi^{\mathrm{c}} (q)|< \left(1- \dfrac{\log{\lambda_{\mathrm{bh}} (B^+) }}{2L^{\mathrm{c}}(\Lambda)} \right) |\chi^{\mathrm{c}} (p)|,
                $
            \item
            $\mathcal{O}(q_f)$ is an $(\epsilon, \rho)$-good approximation of $\mathcal{O}(p_f)$, where
                $
                    \rho\eqdef 1- \dfrac{|\chi^{\mathrm{c}} (p)|}{2L^{\mathrm{c}}(\Lambda)}.
                $
    \end{itemize}
\end{lemma}

The proof of  Lemma~\ref{l.periodicpointshomoclinicclass}  
borrows ingredients from the proof of~\cite[Proposition 3.12]{DiaYanZha:}, where nonhyperbolic measures are constructed for mostly expanding and strongly partially hyperbolic sets with a minimal strong unstable foliation.

Consider any saddle $p \in \Lambda$ with $\chi^\c(p)<0$ and the 
measure $\mu = \mu_{\mathcal{O}(p)}$.
Applying Lemma~\ref{l.periodicpointshomoclinicclass} to  $p$ and arguing inductively, we get:

\begin{coro}
\label{c.periodicpointshomclass}
    There are $C_{f} \in (0,1)$ and
 a sequence of saddles $(q_n)_n$  in $\Lambda$ with $\chi^{\mathrm{c}}(q_n)<0$ such that
\begin{enumerate}
\item
$q_1=p$,
\item
$\pi(q_{n+1})>  \pi (q_n)$, in particular, $\pi(q_n)\to \infty$ as $n\to \infty$,
\item
$|\chi^{\mathrm{c}} (q_{n+1})|< C_{f} |\chi^{\mathrm{c}} (q_n)|$,
in particular, $\chi^c(q_n)\to 0$,
\item
$\mathcal{O}(q_{n+1})$ is an $(\epsilon_n, \rho_n)$-good approximation of $\mathcal{O}(q_n)$, where 
    \begin{equation}
        \epsilon_n  \eqdef  \ \frac{(\min_{1\leq  j \leq n}d_j) \cdot |\chi^{\mathrm{c}}(q_1)|}{  3 \cdot 2^n \cdot L^{\mathrm{c}}(\Lambda)} \qquad \mbox{and} \qquad
        \rho_n  \eqdef 1  - \frac{|\chi^{\mathrm{c}}(q_{n+1})|}{2L^{\mathrm{c}}(\Lambda)},
    \end{equation}
   \item for every $n$, 
    $\displaystyle{ \sum_{k=n}^{\infty} } \epsilon_k < \dfrac{d_n}{3}$, where  $d_i= d_{\mathcal{O}(q_i)}$ defined in Equation~\eqref{e.distorbitaseq}.
\end{enumerate}
\end{coro}

Take $(q_n)_n$ as in Corollary~\ref{c.periodicpointshomclass} and let
$\mathcal{O}_n\eqdef \mathcal{O}(q_n)$ and
 $\mu_n\eqdef \mu_{\mathcal{O}_n}$. 
By construction, the sequence $(\mu_n)_n$ is GIKN, recall Definitions~\ref{d.GIKNseq} and \ref{d.GIKNmes}. By 
Theorem~\ref{t.l.p.strongnonper} and 
 item (3) in Corollary~\ref{c.periodicpointshomclass}, the sequence
$(\mu_n)_n$ converges to some
$\nu \in \mathcal{M}^0_{\mathrm{aGIKN}}(f,\Lambda)$.

 We now estimate  $W_1(\mu,\nu)$. Since $\mu_1=\mu$ and $\mu_n \to \nu$, as $n \to \infty$, it follows that 
    \begin{equation}
    \label{eq.distmunu}
        W_1(\nu,\mu) \leq \sum_{n=1}^{\infty} W_1(\mu_n.\mu_{n+1}).
    \end{equation}

For simplicity, in what is below,  
for $L^\c (\Lambda)$ as in \eqref{e.Lc} and $\lambda_{\mathrm{bh}}(B^+)$ as in Remark~\ref{r.cuexpansion},
we  write 
$$
L^\c =L^\c (\Lambda),
 \qquad {\lambda_{\mathrm{bh}}}={\lambda_{\mathrm{bh}}(B^+)}.
$$

\begin{lemma}
\label{l.seq}
The sequence $(\mu_n)_n$ satisfies
    \begin{equation}
    \label{eq.inequalityinconvergence}
        W_1(\mu_n,\mu_{n+1})<\left( \frac{ \mathrm{diam}(M)}{2^{n-1} \cdot 3 \cdot 2L^{\mathrm{c}}}+K C_{f}^n\right)|\chi^{\mathrm{c}}(\mu)|, \quad  \mbox{where} \quad K \eqdef \frac{\mathrm{diam}(M)}{L^{\mathrm{c}}}.
    \end{equation}
\end{lemma}
\begin{proof}
    By the definition of the Wasserstein distance, we need to estimate
$$
\left| \int \varphi d \mu_{n+1} - \int \varphi d \mu_n \right|, \qquad
\mbox{where  $\varphi \in \mathrm{Lip}(1)$.}
$$ 
Set, for each $n$, and $x \in \Lambda$,
$$
\eta_{\pi_n}(x) \eqdef  \frac{1}{\pi_n}\sum_{i=0}^{\pi_n-1} \delta_{f^i(x)},
\qquad 
\mbox{where $\pi_n$ is the period of $\mathcal{O}_n$.}
$$
Note that, since $\mathcal{O}_{n+1}$ is a periodic orbit, it holds
\begin{align*}
    \sum_{ x \in \mathcal{O}_{n+1}} \int \varphi d \eta_{\pi_n}(x) & = \sum_{ x \in \mathcal{O}_{n+1}} \left( \frac{1}{\pi_n} \sum_{i=0}^{\pi_n-1} \varphi(f^i(x))\right) = \frac{1}{\pi_n} \sum_{ x \in \mathcal{O}_{n+1}} \sum_{i=0}^{\pi_n-1} \varphi(f^i(x)) \\
    & = \frac{1}{\pi_n} \sum_{x \in \mathcal{O}_{n+1}} \, \pi_n \varphi(x) = \pi_{n+1} \, \int \varphi d \mu_{n+1}.
\end{align*}
Recall Definition~\ref{d.goodpoints} and the sets
$
\mathcal{O}_{n+1,\epsilon_n}^{\mathrm{i}}(\mathcal{O}_n)$, $\mathrm{i}\in\{ \mathrm{g}, \mathrm{b}\}$.
 It follows that
\begin{align*}
\left| \int \varphi d \mu_{n+1} - \int \varphi d \mu_n \right| 
& = \frac{1}{\pi_{n+1}}\left| \sum_{ x \in \mathcal{O}_{n+1}} \int \varphi d \eta_{\pi_n}(x) -\pi_{n+1} \int \varphi d\mu_{n} \right| \\
& \leq \frac{1}{\pi_{n+1}}\left| \sum_{ x \in \mathcal{O}_{n+1,\epsilon_n}^{\mathrm{g}}(\mathcal{O}_n)}  \left(\int \varphi d \eta_{\pi_n}(x) - \int \varphi d\mu_{n}\right) \right| + \\
& + \frac{1}{\pi_{n+1}}\left| \sum_{ x \in 
\mathcal{O}_{n+1,\epsilon_n}^{\mathrm{b}}(\mathcal{O}_n)} \left(\int \varphi d \eta_{\pi_n}(x) -\int \varphi d\mu_{n} \right) \right|.
\end{align*}
To estimate the first sum,
take any $x \in \mathcal{O}_{n+1, \epsilon_n}^{\mathrm{g}}(\mathcal{O}_n)$. 
As $\varphi \in \mathrm{Lip}(1)$, it follows that 
\[
\left| \int \varphi d\eta_{\pi_n}(x) - \int \varphi d \mu_n \right|<  \epsilon_n= \frac{(\min_{j \leq n} d_j) |\chi^{\mathrm{c}|}(\mu)}{2^{n-1} \cdot 3 \cdot 2{L^{\mathrm{c}}}} 
< \frac{ \mathrm{diam}(M) |\chi^{\mathrm{c}|}(\mu)}{2^{n-1} \cdot 3 \cdot 2{L^{\mathrm{c}}}}.
\]
Therefore,
\begin{equation}\label{eq.first}
    \frac{1}{\pi_{n+1}}\left| \sum_{ x \in 
\mathcal{O}_{n+1,\epsilon_n}^{\mathrm{g}}(\mathcal{O}_n)}  \left(\int \varphi d \eta_{\pi_n}(x) - \int \varphi d\mu_{n}\right) \right| <\frac{ \mathrm{diam}(M) | \chi^{\mathrm{c}}(\mu)|}{2^{n-1} \cdot 3 \cdot 2{L^{\mathrm{c}}}}.
\end{equation}
To estimate the second sum, write
\begin{equation}
\label{eq.second}
\begin{split}
     \frac{1}{\pi_{n+1}}&\left| \sum_{x \in 
\mathcal{O}_{n+1,\epsilon_n}^{\mathrm{b}}(\mathcal{O}_n)} \left(\int \varphi d \eta_{\pi_n}(x) -\int \varphi d\mu_{n} \right) \right|\\
     & \qquad < \frac{\#(\mathcal{O}_{n+1,\epsilon_n}^{\mathrm{b}}(\mathcal{O}_n)}
     {\pi_{n+1}} \left| \int \varphi d \eta_{\pi_n}(x) - \int \varphi d\mu_{n}  \right| \\
    & \qquad  < 2 \cdot \mathrm{diam}(M) \frac{|\chi^{\mathrm{c}}(\mu_n)|}{2L^{\mathrm{c}}}=  K| \chi^{\mathrm{c}}(\mu_n) |,
\end{split}
\end{equation}
 where in the last inequality we use that $\mathcal{O}_{n+1}$ is a $(\epsilon_n,\rho_n)$-good approximation of $\mathcal{O}_{n}$ and that $\|\varphi\|\leq \mathrm{diam}(M)$, and in the equality the definition of $K$ in \eqref{eq.inequalityinconvergence}.

Putting together the estimates in~\eqref{eq.first} and \eqref{eq.second}, we obtain
\[
\begin{split}
\left| \int \varphi d\mu_{n+1} - \int \varphi d \mu_n \right| &< \frac{ \mathrm{diam}(M) \cdot \chi^{\mathrm{c}}(\mu)}{2^{n-1} \cdot 3 \cdot 2L^{\mathrm{c}}}+ K |\chi^{\mathrm{c}}(\mu_n)|
\end{split}
\]
Using that $|\chi^{\mathrm{c}}(\mu_n)|< C_f^n|\chi^{\mathrm{c}}(\mu)|$, recall (3) in Corollary~\ref{c.periodicpointshomclass}, we get
$$
W_1(\mu_n,\mu_{n+1})<\left( \frac{ \mathrm{diam}(M)}{2^{n-1} \cdot 3 \cdot 2L^{\mathrm{c}}}+K C_{f}^n\right) |\chi^{\mathrm{c}}(\mu)|,
$$
ending  the proof of the lemma.    
\end{proof}

To finish the proof of the  proposition, note that by Lemma~\ref{l.seq} and \eqref{eq.distmunu},   
 $$
        W_1(\mu,\nu) \leq \sum_{n=1}^{\infty} W_1(\mu_n,\mu_{n+1})<\sum_{n=1}^{\infty}\left( \frac{ \mathrm{diam}(M)}{2^{n-1} \cdot 3 \cdot 2L^{\mathrm{c}}}+K C_{f}^n\right)|\chi^{\mathrm{c}}(\mu)|,
    $$
Since $0<C_{f}<1$, it follows that
    $$
        A_\Lambda \eqdef \sum_{n=1}^{\infty}\left( \frac{ \mathrm{diam}(M)}{2^{n-1} \cdot 3 \cdot 2L^{\mathrm{c}}}+K C_{f}^n\right) < \infty,
    $$
where $A_\Lambda$ does not depend on  $\mu$. This ends the proof of the proposition. 
\qed

\subsection{Proof of Lemma~\ref{l.periodicpointshomoclinicclass}}
\label{ss.proofofperiodicpointshomoclinicclass}

Let $f$ and $\Lambda$ be as in Proposition~\ref{p.mainprop}.  
Note that as $\Lambda$ is an isolated set, there is a neighborhood $U$ of $\Lambda$ such that 
$\Lambda=\Lambda_f(\overline U)=\Lambda_f (U)$, recall
\eqref{e.LambdaU}.

\begin{remark}
\label{r.orbitcontained}
    If $\mathcal{O}(x) \subset U$, then $\mathcal{O}(x) \subset \Lambda_f(U)=\Lambda$.
\end{remark}

\subsubsection{Liao-Gan's shadowing lemma}
\label{ss.shadowinglemma}

We will use  the shadowing lemma in~\cite{Gan:02} that we state for a strongly partially hyperbolic set $\Lambda$, considering the splitting $T_{\Lambda}M = E \oplus F$, where $E = E^\s \oplus E^\c \eqdef E^{\cs}$ and $F=E^\u$. 

Given $x \in \Lambda$ and $n \in \bN$, denote by $[x,n]$ the {\em{orbit segment}} $\{x,f(x),\dots,f^n(x)\}$.

\begin{lemma}[\cite{Gan:02}]
\label{l.gan}
    Let $f \in \mathrm{Diff}^1(M)$ and $\Lambda$ be a strongly partially hyperbolic set with splitting $T_{\Lambda}M = E \oplus F$. Given $\lambda<0$, there are $L,a_0>0$ such that for every $a \in (0,a_0]$ and every 
     orbit segment $[x,n]$ such that 
        \begin{enumerate}
            \item $d(x,f^n(x))<a$;
            \item $\frac{1}{k} \sum_{i=0}^{n-1}\log{\|Df|_{E^{\cs}(f^i(x))}\|} \leq \lambda$ for every $1 \leq k \leq n$; and
            \item $\frac{1}{k} \sum_{i=0}^{n-1}\log{\|Df^{-1}|_{E^\u (f^{n-i}(x))}\|} \leq \lambda$ for every $1 \leq k \leq n$;
        \end{enumerate}
    there exists a periodic point $p$ with period $\pi(p)=n$ such that 
        $$
            d(f^i(x),f^i(p)) < L  a, \quad \mbox{for every } i \in \{0,\dots,n\}.
        $$
\end{lemma}

\begin{remark}
\label{r.lambdau1}
    Let $[x,n]$ be an orbit segment contained in $\Lambda$. By the partial hyperbolicity of $\Lambda$, there are
    $C>0$ and ${\lambda_\u} \in (0,1)$ such that 
        \begin{equation*}
            \|Df^{-i}|_{E^\u(f^i(x))}\|\leq C i \log{\lambda_\u}<0, \qquad \mbox{for every $1 \leq i \leq n$}.
        \end{equation*}
    By~\cite{Gou:07}, there is an {\em{adapted metric}} such that $C=1$. Hence, 
   Condition $(3)$ in Lemma~\ref{l.gan} is always satisfied for $\lambda=\log{\lambda_\u}$.
    \end{remark}

\subsubsection{End of the proof of Lemma~\ref{l.periodicpointshomoclinicclass}}
The number $\epsilon_0>0$ in the statement of the lemma is chosen as follows:
every $x$ with $d(x,\Lambda)<\epsilon_0$ belongs to $U$.              
In what follows, fix any $\epsilon \in (0,\epsilon_0)$
and take any saddle $p\in \Lambda$ with $\chi^\c (p)<0$. For notational  
simplicity, let us assume that $f(p)=p$.
 We construct an orbit segment satisfying Conditions (1)-(3) of Lemma~\ref{l.gan}
for appropriate constants to be  fixed.

Recall the map $\varphi^\c$ in~\eqref{eq.derivadacentrala}. By continuity, 
given 
\begin{equation}
\label{eq.choiceofdelta}
       0< \delta < \frac{|\le^{\mathrm{c}}(p)|}{4},
\end{equation}
there is $\gamma>0$ such that 
\begin{equation}
\label{eq.choiceofgamma}
    |\varphi^\mathrm{c}(x) - \varphi^{\mathrm{c}}(y)|<\delta \qquad \mbox{for every $x,y \in \Lambda$ such that $d(x,y)<\gamma$}.
\end{equation}

Recalling Remark~\ref{r.lambdau1}, Condition (3) of Lemma~\ref{l.gan} is satisfied for 
    \begin{equation}
    \label{eq.choiceoflambda}
        \lambda \eqdef \max  \left\{ \frac{\le^{\mathrm{c}}(p)}{4},\log{\lambda_\u}\right\}<0.
    \end{equation}
Let $L,a_0$ be the numbers associated to $\lambda$ given by Lemma~\ref{l.gan} and take
    \begin{equation}
    \label{eq.escolhahomclass}
        0<a<\min \left\{a_0,\frac{\gamma}{L+1},\frac{\epsilon}{L+1}\right\}.
    \end{equation}

We are now ready to construct an orbit segment.
By the hypotheses and Remark~\ref {r.refinedflipfop},
and after shrinking $a$ if necessary and taking small $\rho>0$,
there is a split flip-flop configuration
$$
 \mathrm{SFF}(\mathfrak{D}_{a/4, \rho}(p),\mathfrak{D}_{\mathrm{bet}}(B^+),N,C), 
 \qquad
N=N_{a/4,\rho, p,B^+}, \quad C=C_{a/4,\rho, p,B^+}.
$$
We can assume that the disks in $\mathfrak{D}_{a/4, \rho}(p)$ are contained in the 
ball of radius $\gamma/2$ centered at $p$. For simplicity, write 
$\mathfrak{D}(p)=\mathfrak{D}_{a/4, \rho}(p)$.

 Note that for every disk $D' \in \mathfrak{D}(p)$,
the set $f^C(D')$ contains a disk $D\in \mathfrak{D}_{\mathrm{bet}}(B^+)$.
 By Lemma~\ref{l.sobrediscos},   for every $k \in \N$
 the disk $D$ contains a subdisk $D_k$ such that 
\begin{itemize}
    \item $f^i(D_k)$ is contained in a disk in $\mathfrak{D}_{\mathrm{bet}}(B^+)$ for every $i \in \{0,\dots,k-1\}$; and 
    \item $f^k(D)$ contains a disk $\widetilde D \in  \mathfrak{D}_{\mathrm{bet}}(B^+)$.   
    \end{itemize}
Note that  $f^N(\widetilde{D})$ contains a disk in $\mathfrak{D}(p)$ that intersects $W^\mathrm{s}(p)$ transversally at some point $y$. 
Define 
    \begin{equation}
     \label{eq.constrqk}
        q_k \eqdef f^{-N-k}(y), \qquad   q_k \in f^{-k}(\widetilde{D}) \subset D_k \subset D \subset f^{C}(D').
    \end{equation}
Hence $f^{-C}(q_k) \in D'$.
    Since $p$ is a fixed point, it follows that
    \begin{equation}
    \label{eq.qkmtop}
        q_{k,\ell} \eqdef f^{-\ell-C}(q_k) \in W^{\mathrm{u}}_{a/4}(p), \quad \mbox{for every $\ell \in \N$}.
    \end{equation}
Thus, the orbit segment $[q_{k,\ell}, \ell+C+k+N]$ satisfies Condition (1).

\begin{remark}
\label{r.positionoforbit}
    By~\eqref{eq.qkmtop},
           $ f^n(q_{k,\ell}) \in W^{\mathrm{u}}_{a/4}(p)$  for every $0 \leq n \leq \ell$.
    Hence, by~\eqref{eq.constrqk},  $f^n(q_{k,\ell})$ belongs to some 
    disk in $\mathfrak{D}_{\mathrm{bet}}(B^+)$ for every $\ell+ C \leq n \leq \ell+ C +k$. Since the blender-horseshoe $B^+$ is unstable, the disks in-between are tangent to $\mathcal{C}^\u$, and since $\mathcal{C}^\u \subset \mathrm{int}(\mathcal{C}^\cu)$, recalling Remark~\ref{r.cuexpansion} it follows that    
    $$
            \varphi^\c(f^n(q_{k,\ell})) > \log{\lambda_{\mathrm{bh}}}
            \qquad \mbox{for every $\ell+C \leq n \leq \ell+ C +k$.}    
        $$   
\end{remark}

We claim that $[q_{k,\ell},\ell+C+k+N]$ satisfies Condition (2). Due to the domination and 
since we are considering an adapted metric, it is enough to check it
 for the subbundle $E^{\mathrm{c}}$ of $E^{\cs}$:

\begin{claim}
\label{af.cond2}
        There are arbitrarily large $\ell$ and $k$ such that 
            $$
                \frac{1}{n} \sum_{i=0}^{n-1} \varphi^{\mathrm{c}}(f^i(q_{k,\ell})) < \frac{\chi^{\mathrm{c}}(p)}{2} \leq \lambda,
                \qquad \mbox{for every $1 \leq n \leq \ell+C+k+N$.}
            $$
\end{claim}

\begin{proof}
        We split the proof in two cases: $1 \leq n \leq \ell$ and $\ell+1 \leq n \leq \ell+C+k+N$. 
    
    \begin{fact}
    \label{fact.estimativaprimeiraparte}
            For every $1 \leq n \leq \ell$, it follows that 
                \begin{equation*}
                    \left| \frac{1}{n} \sum_{i=0}^{n-1} \varphi^{\mathrm{c}}(f^i(q_{k,\ell})) - \chi^{\mathrm{c}}(p) \right| < \delta, \quad \mbox{where $\delta$ is as in \eqref{eq.choiceofdelta}.}
                \end{equation*}
    \end{fact}
    \begin{proof}
        Fix  $1 \leq n \leq \ell$. By Remark~\ref{r.positionoforbit} and recalling the choice of $a$ in~\eqref{eq.escolhahomclass}, we have  that $f^n(q_{k,\ell}) \in W^{\mathrm{u}}_{a/4}(p)$, therefore
                $$
                    d(f^n(q_{k,\ell}),p)< \frac{a}{4}< \gamma.
                $$
        By the choice of $\gamma$ in~\eqref{eq.choiceofgamma}, we get
                \begin{equation}
                \label{eq.fnqkmtop}
                    | \varphi^{\mathrm{c}}(f^n(q_{k,\ell}))-\varphi^{\mathrm{c}}(p)|< \delta, \qquad
                    \mbox{for every $1 \leq n \leq \ell$.}
                \end{equation}
          The fact follows noting that $p$ is a fixed point and $\chi^\mathrm{c}(p) = \varphi^{\mathrm{c}}(p)$.
    \end{proof}

    By Fact~\ref{fact.estimativaprimeiraparte}, the choice of $\delta$ in~\eqref{eq.choiceofdelta}, and $\le^{\mathrm{c}}(p)<0$,  it follows that 
        \begin{equation*}
            \frac{1}{n} \sum_{i=0}^{n-1} \varphi^{\mathrm{c}}(f^i(q_{k,\ell}))< \chi^{\mathrm{c}}(p) + \delta < \frac{\chi^{\mathrm{c}}(p)}{2},
            \quad \mbox{for every $1 \leq n \leq \ell$},
        \end{equation*}
        proving the claim for $1 \leq n \leq \ell$.
    
   We now consider the case $\ell+1 \leq n \leq \ell+C+k+N$, where we
    assume that $\ell, k$ are large enough satisfying
      \begin{equation}
    \label{eq.escmk}
        \frac{2L^{\mathrm{c}} +2 \delta -\chi^{\mathrm{c}}(p)}{-\chi^{\mathrm{c}}(p) -2 \delta} < \frac{\ell}{k} < \frac{2L^{\mathrm{c}} +3 \delta -\chi^{\mathrm{c}}(p)}{-\chi^{\mathrm{c}}(p) -3 \delta}, \qquad 
        \mbox{$L^\c$ as in \eqref{e.Lc}.}
    \end{equation}

     Using Fact~\ref{fact.estimativaprimeiraparte} and the definition of $L^\c $  it follows that
            \begin{equation}
            \label{eq.secondcase1}
            \begin{split}
                \sum_{i=0}^{n-1} \varphi^{\mathrm{c}}(f^i(q_{k,\ell})) 
                & = \sum_{i=0}^{\ell-1} \varphi^{\mathrm{c}}(f^i(q_{k,\ell})) + \sum_{i=\ell}^{n-1} \varphi^{\mathrm{c}}(f^i(q_{k,\ell})) \\
                & \leq \ell  (\chi^{\mathrm{c}}(p) + \delta) + (n-\ell)  L^{\mathrm{c}}
                \\
                &= n L^{\mathrm{c}} + {\ell} \left(\chi^{\mathrm{c}}(p) + \delta - L^{\mathrm{c}}\right)
            \end{split}
            \end{equation}
    Since
    $\chi^{\mathrm{c}}(p) + \delta - L^{\mathrm{c}} <0$,
    dividing both sides of ~\eqref{eq.secondcase1} by $n$,
    we get:
        \begin{equation}
        \label{eq.estiamtiva}
        \begin{split}
            \frac{1}{n} \sum_{i=0}^{n-1} \varphi^{\mathrm{c}}(f^i(q_{k,\ell})) 
            & = L^{\mathrm{c}}+ \frac{\ell}{n}  \left(\chi^{\mathrm{c}}(p) + \delta - L^{\mathrm{c}}\right) \\
            & \leq L^{\mathrm{c}}  + \frac{\ell}{\ell+C+k+N}  \left(\chi^{\mathrm{c}}(p) + \delta - L^{\mathrm{c}}\right) \\
            & = \frac{\ell}{\ell+C+k+N}  (\chi^{\mathrm{c}}(p) + \delta) + \frac{C+k+N}{\ell+C+k+N} L^{\mathrm{c}}.
        \end{split}
        \end{equation}
        Taking $k$ sufficiently large, we get 
        \begin{equation}
        \label{eq.estimativa2}
        \begin{split}
             \frac{1}{n} \sum_{i=0}^{n-1} \varphi^{\mathrm{c}}(f^i(q_{k,\ell})) & \leq \frac{(\ell/k) (\chi^{\mathrm{c}}(p) + \delta) }{\ell/k+1+(C+N)/k}  + \frac{(1+(C+N)/k) L^{\mathrm{c}}}{\ell/k+1+(C+N)/k}   \\
             & \leq \frac{\ell/k}{\ell/k+1}  (\chi^{\mathrm{c}}(p) + 2 \delta) + \frac{1}{\ell/k+1}  L^{\mathrm{c}}\\
             & \leq \frac{\ell/k}{\ell/k+1} \chi^{\mathrm{c}}(p) + \frac{1}{\ell/k+1} L^{\mathrm{c}}+ 2 \delta\\
              & =  \frac{\ell}{\ell+k} \left( \chi^{\mathrm{c}}(p) - L^{\mathrm{c}}\right) +L^{\mathrm{c}} +2 \delta.
       \end{split}
        \end{equation}

We now  estimate $\frac{\ell}{\ell+k}$.
 Using~\eqref{eq.escmk}, we get
        \begin{equation} 
        \label{eq.ellell+k}
        \begin{split}
         \frac{\ell+k}{\ell}&=     1 + \frac{k}{\ell} < \frac{-\chi^{\mathrm{c}}(p) -2 \delta}{2L^{\mathrm{c}} +2 \delta -\chi^{\mathrm{c}}(p)} +1 = \frac{2L^\c - 2 \chi^\c(p)}{2L^\c+2\delta -\chi^\c(p)}\\
       &
         > \frac{2L^\c +2\delta -\chi^\c(p)}{2L^\c - 2 \chi^\c(p)}.
        \end{split}
        \end{equation}   
For further use, arguing analogously, we obtain
    \begin{equation}
    \label{eq.lowerboundkmk}
        \frac{k}{\ell+k} = \left(1+ \frac{\ell}{k}\right)^{-1} > \left(\frac{2L^{\mathrm{c}}+3\delta-\chi^{\mathrm{c}}(p)}{-\chi^{\mathrm{c}}(p)-3\delta} +1\right)^{-1} = \frac{-\chi^{\mathrm{c}}(p)-3\delta}{2L^{\mathrm{c}}-2\chi^{\mathrm{c}}(p)}.
    \end{equation}

    Since $\chi^{\mathrm{c}}(p)-L^{\mathrm{c}}<0$ and $\delta<\frac{|\chi^{\mathrm{c}}(p)|}{4}$, 
    using \eqref{eq.estimativa2} and \eqref{eq.ellell+k},
    it follows 
    \begin{equation}
    \begin{split}
    \label{eq.estimativasegundaparte}
        \frac{1}{n}\sum_{i=0}^{n-1} \varphi^{\mathrm{c}}(f^i(q_{k,\ell})) 
        & \leq \frac{\ell}{\ell+k}  ( \chi^{\mathrm{c}}(p) - L^{\mathrm{c}}) + L^{\mathrm{c}} + 2 \delta \\
        & \leq \frac{2L^{\mathrm{c}} +2 \delta -\chi^{\mathrm{c}}(p)}{2L^{\mathrm{c}}
        -2\chi^{\mathrm{c}}(p)}  (\chi^{\mathrm{c}}(p) -L^{\mathrm{c}}) +L^{\mathrm{c}} +2 \delta \\
     { \footnotesize{\mbox{(choice of $\lambda$ in~\eqref{eq.choiceoflambda})}}}  
        & = \frac{\chi^c(p)}{2} + \delta < \frac{\chi^c(p)}{2} - \frac{\chi^c(p)}{4}< \frac{\chi^c(p)}{4}< \lambda,
    \end{split}        
    \end{equation}
ending the proof of the claim.  
\end{proof}

The next claim provides  the estimate of the $\chi^\c(q)$ in Lemma~\ref{l.periodicpointshomoclinicclass}.

\begin{claim}
\label{cl.semlabel}
Let $\ell, k$ be the (large) numbers in Claim~\ref{af.cond2} and 
$q_{k,\ell}$ the point in \eqref{eq.qkmtop}.
There is a periodic point
  $q \in \Lambda$ with period $\pi(q)=\ell+C+k +N$ such that
    \begin{equation} 
    \label{eq.shadows}
        d(f^i(q_{k,\ell}),f^i(q))< L a, \quad \mbox{for every $i \in \{0,\dots,\pi(q)$,  where $a$ is as in~\eqref{eq.escolhahomclass}}
    \end{equation}
whose central Lyapunov exponent satisfies
$$
\chi^\c (p)  \left( \frac{2 L^\c - \log {\lambda_{\mathrm{bh}}}}{2  L^\c } \right)
\leq \chi^\c (q) \leq \frac{\chi^\c(p)}{4}.
$$
\end{claim}

\begin{proof}
As the orbit segment $[q_{k,\ell}, \ell+C+k+N]$ satisfies all conditions 
of Lemma~\ref{l.gan},  there exists a periodic point $q$ with period
$\pi(q)= \ell+C+k+N$ satisfying  \eqref{eq.shadows}. To see that $q\in \Lambda$,
recall the choices of $\epsilon_0$ and $a$. By~\eqref{eq.shadows}, we get that 
        $$
            d(f^i(q),\Lambda) < \epsilon_0, \quad \mbox{for every } i \in \{0,\dots,\pi(q)\}.
            $$ 
  Hence, $\mathcal{O}(q) \subset U$ and,   by Remark~\ref{r.orbitcontained},  $q\in \Lambda$.       
            
         To estimate the Lyapunov exponent $\chi^\c(q)$, we
    estimate $\varphi^{\mathrm{c}}(f^n(q))$.

\begin{fact}
\label{fact.primeirapartedeq}
    Let $\gamma$ as in~\eqref{eq.choiceofgamma}. Then
       $d(f^n(p),f^n(q))< \gamma$  for every $1 \leq n \leq \ell-1$.
\end{fact}
\begin{proof}
 Recall the choice of $a$ in \eqref{eq.escolhahomclass}.
    Using Remark~\ref{r.positionoforbit} and that $f(p)=p$, one gets
    \begin{equation}
    \label{eq.inicioorbita}
        d(f^n(q_{k,\ell}),f^n(p))= d(f^n(q_{k,\ell}),p) < a \quad \mbox{for every $1 \leq n \leq \ell-1$}.
    \end{equation}
Using~\eqref{eq.shadows} and~\eqref{eq.inicioorbita}, we get that, for every $0 \leq n \leq \ell-1$, it holds
    \begin{equation}
    \begin{split}
    \label{eq.aproxprimeiraparte}
        d(f^n(p),f^n(q)) & < d(f^n(q),f^n(q_{k,\ell})) + d(f^n(q_{k,\ell}),f^n(p)) \\
        & < L a + a =  (L+1)a < \gamma,
    \end{split}
    \end{equation}
proving the fact.
\end{proof}

By Fact~\ref{fact.primeirapartedeq}, the choice of $\gamma$ in~\eqref{eq.choiceofgamma}, it follows 
    \begin{equation}
    \label{eq.contrprimeirapartedeq}
        \varphi^{\mathrm{c}}(f^n(q)) < \varphi^{\mathrm{c}}(p) + \delta = \chi^{\mathrm{c}}(p) + \delta,
       \qquad \mbox{for every $0 \leq n \leq \ell$.}
    \end{equation}

Using the definition of $L^{\mathrm{c}}$ and Equation~\eqref{eq.contrprimeirapartedeq}, we get 
    \begin{equation*}
    \label{eq.controledeexpq1}
    \begin{split}
        \chi^{\mathrm{c}}(q) & =  \frac{1}{\pi(q)} \sum_{i=0}^{\pi(q)-1}\varphi^{\mathrm{c}}(f^i(q))  \\
         &\leq \frac{\ell}{\ell+C+k+N}  (\chi^{\mathrm{c}}(p) +\delta) + \frac{k+C+N}{\ell+C+k+N}  L^{\mathrm{c}} \\
        & = \frac{\ell/k}{\ell/k+1+(C+N)/k}  (\chi^{\mathrm{c}}(p) +\delta) + \frac{1+(C+N)/k}{\ell/k+1+(C+N)/k}  L^{\mathrm{c}}\\
      { \footnotesize{\mbox{($k$ large as in~\eqref{eq.estimativa2})}}}  
      & \leq \frac{\ell/k}{\ell/k+1}  (\chi^{\mathrm{c}}(p) +2\delta) + \frac{1}{\ell/k+1}  L^{\mathrm{c}}\\
         {\footnotesize{\mbox{($\ell,k$ as in~\eqref{eq.escmk})}}}  
               & \leq \frac{\ell}{\ell+k}  \chi^{\mathrm{c}}(p) + \frac{k}{\ell+k} L^{\mathrm{c}} + 2\delta< \frac{\chi^{\mathrm{c}}(p)}{4},
    \end{split}
    \end{equation*}
where in the last inequality we argue as in~\eqref{eq.estimativasegundaparte}.
 This gives the 
upper bound. 

We now get the lower bound. By Fact~\ref{fact.primeirapartedeq} and the choice of $\gamma$ in~\eqref{eq.choiceofgamma}, 
        \begin{equation*}
            \varphi^{\mathrm{c}}(f^n(q)) \geq \varphi^{\mathrm{c}}(p) - \delta = \chi^{\mathrm{c}}(p) - \delta
            \quad \mbox{for every $0 \leq n \leq \ell-1$.}
        \end{equation*}

Additionally, recalling Remark~\ref{r.positionoforbit}, we have that
    \begin{equation*}
    \label{eq.derivsegundaparte}
        \varphi^{\mathrm{c}}(f^n(q)) > \log{\lambda_{\mathrm{bh}}}  \quad \mbox{for every $\ell+C \leq n \leq \ell+C+k$}.
    \end{equation*}

Therefore, taking $\ell, k$ large enough and arguing as above, we get
    \begin{equation}
    \begin{split}
    \label{eq.lowerbound}
        \chi^{\mathrm{c}}(q)& = \frac{1}{\ell+C+k+N}\sum_{i=0}^{\ell+C+k+N-1} \varphi^{\mathrm{c}}(f^i(q)) \\
        &\geq \frac{\ell  ( \chi^{\mathrm{c}}(p) - \delta)}{\ell+C+k+N}  + \frac{k  \log{\lambda_{\mathrm{bh}}}}{\ell+C+k+N}
         - \frac{(C+N)  L^{\mathrm{c}}}{\ell+C+k+N}  \\
        & \geq \frac{\ell}{\ell+k} \chi^{\mathrm{c}}(p) + \frac{k}{\ell+k} \log{\lambda_{\mathrm{bh}}} -2 \delta\\
        & = \chi^{\mathrm{c}}(p) + \frac{k}{\ell+k} (\log{\lambda_{\mathrm{bh}}} - \chi^{\mathrm{c}}(p)) - 2 \delta \\
      { \footnotesize{\mbox{(by \eqref{eq.lowerboundkmk})}}}      & \geq \chi^{\mathrm{c}}(p) + \frac{(-\chi^{\mathrm{c}}(p)-3 \delta)}{2L^{\mathrm{c}}-2\chi^{\mathrm{c}}(p)} (\log{\lambda_{\mathrm{bh}}}-\chi^{\mathrm{c}}(p))-2\delta.
    \end{split}
\end{equation}

From $L^{\mathrm{c}} \geq \log{\lambda_{\mathrm{bh}} }$,
we get  $L^\c (\chi^\c(p))^2 \geq \log{\lambda_{\mathrm{bh}}}(\chi^\c(p))^2$ and therefore
\begin{equation*}
            \frac{(-\chi^{\mathrm{c}}(p))(\log{\lambda_{\mathrm{bh}}}-\chi^{\mathrm{c}}(p))}{2L^{\mathrm{c}}-2\chi^{\mathrm{c}}(p)}  \geq \frac{\log{\lambda_{\mathrm{bh}}}\cdot(-\chi^{\mathrm{c}}(p))}{2L^{\mathrm{c}}}.
\end{equation*}
Hence, by shrinking $\delta$ if necessary, we get
    \begin{equation}
    \label{e.deltadelta}
            \frac{(-\chi^{\mathrm{c}}(p)-3 \delta)(\log{\lambda_{\mathrm{bh}}}-\chi^{\mathrm{c}}(p))}{2L^{\mathrm{c}}-2\chi^{\mathrm{c}}(p)} -2\delta \geq \frac{\log{\lambda_{\mathrm{bh}}}\cdot(-\chi^{\mathrm{c}}(p))}{2L^{\mathrm{c}}}.
    \end{equation}

Taking $\delta$ as in~\eqref{e.deltadelta} and  using~\eqref{eq.lowerbound} we conclude that 
    $$
        \chi^{\mathrm{c}}(q) \geq \chi^{\mathrm{c}}(p) + \frac{\log{\lambda_{\mathrm{bh}}}(-\chi^{\mathrm{c}}(p))}{2L^{\mathrm{c}}} = \chi^{\mathrm{c}}(p)\left(1 - \frac{\log{\lambda_{\mathrm{bh}}}}{2L^{\mathrm{c}}}\right),
    $$
getting the lower bound. This  ends the proof of the claim.
\end{proof}

The next claim states the good approximation property in Lemma~\ref{l.periodicpointshomoclinicclass}.   

\begin{claim}
    The orbit $\mathcal{O}(q)$ is an $(\epsilon,1 + \frac{\chi^{\mathrm{c}}(p)}{2 L^{\mathrm{c}}})$-good approximation for $\mathcal{O}(p)$.
\end{claim}
    \begin{proof}
        By~\eqref{eq.aproxprimeiraparte} and by the choice of $a$ in \eqref{eq.escolhahomclass} we have 
                 \begin{equation*}
            \begin{split}
                d(f^n(q),f^n(p)) < (L+1)  a < \epsilon, \quad \mbox{for every $0 \leq n \leq \ell-1$}.
            \end{split}
            \end{equation*}     
     Since $\pi(q)=\ell+C+k+N$, it follows that 
        \begin{equation}
\label{e.countgood}
        \frac{\#\left(\mathcal{O}(q)\right)_{\epsilon}^{\mathrm{g}}\left(\mathcal{O}(p)\right)}{\#\left(\mathcal{O}(q)\right)}  \geq \frac{\ell}{\ell+C+k+N} = 1- \frac{k+C+N}{\ell+C+k+N}.
    \end{equation}
    We now obtain an upper bound for $(k+C+N)/(\ell+C+k+N)$. 
    \begin{fact} 
    \label{f.af.lemaboaap} It holds
        $\displaystyle{\frac{k+C+N}{\ell+C+k+N} < \frac{-\chi^{\mathrm{c}}(p)}{2L^{\mathrm{c}}} +O\left(\frac{1}{k}\right).}$
    \end{fact}
    \begin{proof}
    Since $C+N>0$, it follows that
        \begin{equation*}
        \label{eq.estimativak+nm+k+N}
            \frac{k+C+N}{\ell+C+k+N}<\frac{k+C+N}{\ell+k} = \frac{k}{\ell+k} + \frac{C+N}{\ell+k}
        \end{equation*}
    Using~\eqref{eq.escmk} one has that 
    \begin{equation*}
        1+ \frac{\ell}{k} > 1 + \frac{2L^{\mathrm{c}}+2\delta-\chi^{\mathrm{c}}(p)}{-\chi^{\mathrm{c}}(p)-2\delta} = \frac{2L^{\mathrm{c}}- 2\chi^{\mathrm{c}}(p)}{-\chi^{\mathrm{c}}(p) - 2\delta}.
    \end{equation*}
    Therefore,
    \begin{equation*}
    \label{eq.estimativakm+k}
    \begin{split}
        \frac{k}{\ell+k} = \left(1+ \frac{\ell}{k}\right)^{-1}  < \frac{-\chi^{\mathrm{c}}(p)-2\delta}{2L^{\mathrm{c}}-2\chi^{\mathrm{c}}(p)}< \frac{-\chi^{\mathrm{c}}(p)}{2L^{\mathrm{c}}},
        \end{split}
    \end{equation*}
    which implies the claim.
    \end{proof}
  
Using Fact~\ref{f.af.lemaboaap} and \eqref{e.countgood}, and taking $k$ sufficiently large, we obtain
    \begin{equation*}
        \frac{\#(\mathcal{O}(q))_{\epsilon}^{\mathrm{g}}(\mathcal{O}(p))}{\#(\mathcal{O}(q))} \geq 1 +  \frac{\le^{\mathrm{c}}(p)}{2L^{\mathrm{c}}}
    \end{equation*}
    proving the claim.
    \end{proof}
The proof of Lemma~\ref{l.periodicpointshomoclinicclass} is now complete \hfill \qed

\section{Split flip-flop configurations inside homoclinic classes} 
\label{s.splitflipflopinsidehomoclinicclass}

Recall the residual  set $\mathcal{R}_0(M)$ of $\mathrm{Diff}^1(M)$ in  Theorem~\ref{t.residual}. 
The goal of this section is to prove the following:

\begin{teo}
\label{t.classeblender}
Consider  $f \in \mathcal{R}_0(M)$ and a $C^1$-neighborhood $\mathcal{U}_f$  of $f$ such that 
there are continuous families of saddles $(p_g)_{g \in \mathcal{U}}$ and  $(q_g)_{g \in \mathcal{U}}$
with
\[
\mathrm{s}\text{-}\mathrm{ind}(p_g) > \mathrm{s}\text{-}\mathrm{ind}(q_g)
\]
whose homoclinic classes satisfy
\[
H_g \eqdef H(p_g) = H(q_g)\in \mathrm{ISPH}(g) \qquad \mbox{for every  $g \in \mathcal{U}_f \cap \mathcal{R}_0(M)$}.
\]
Then there exists a locally residual set
$
\mathcal{R}_{\mathcal{U}_f}$ of $ \mathcal{U}_f$
such that, for every $g \in \mathcal{R}_{\mathcal{U}_f}$, the class $H_g$
contains both an unstable blender-horseshoe $B_g^+$ and a stable
blender-horseshoe $B_g^-$ such that
 every saddle in $H_g$ has a split flip--flop configuration with
either $B_g^+$ or $B_g^-$, depending on its $\mathrm{s}$-index.
\end{teo}

This section is organized as follows. In Section~\ref{ss.heterodimensionalcycles}, we define \emph{heterodimensional cycles} and state Proposition~\ref{p.blender2}, which yields blender-horseshoes from such cycles. In Section~\ref{ss.proofoftheoremclassblender}, we prove Theorem~\ref{t.classeblender} by applying this proposition. Finally, Section~\ref{ss.proofofpropblender2} contains the proof of Proposition~\ref{p.blender2}.

\subsection{Heterodimensional cycles}
\label{ss.heterodimensionalcycles}
A diffeomorphism $f$ has a {\em{heterodimensional cycle}} (or  a {\em{cycle}} for short) associated to a pair of saddles
$p$ and $q$ having different $\mathrm{s}$-index if the invariant manifolds of these saddles meet cyclically:
  $$
  W^{\mathrm{s}}(\mathcal{O}(p)) \cap W^{\mathrm{u}}(
  \mathcal{O}(q)) \neq \emptyset \neq W^{\mathrm{u}}(\mathcal{O}(p)) \cap W^{\mathrm{s}}(\mathcal{O}(q)).
  $$
This cycle has {\em{co-index one}} if the $\mathrm{s}$-indices of $p$ and $q$ differ by one.

Suppose now  that $\mathrm{s}$-indices of $p$ and $q$ are $s+1$ and $s$, respectively. Let $\ast \in \{p,q\}$ and $|\lambda_1(\ast)| \leq \dots \leq |\lambda_{\mathrm{dim}(M)}(\ast)|$
be the eigenvalues of $Df^{\pi(\ast)}(\ast)$ enumerated with multiplicity. 
The cycle has {\em{real central eigenvalues}} if 
 $|\lambda_{s+1}(p)| > |\lambda_s(p)|$  and $|\lambda_{s+2}(q)|>|\lambda_{s+1}(q)|$.
 In such a case,
 there is a unique {\em{central eigenvalues}} of $p$ and $q$ denoted by
   $\lambda_{\mathrm{c}}(p)\eqdef \lambda_{s+1}(p)$ 
 and $\lambda_{\mathrm{c}}(q)\eqdef \lambda_{s+1}(q)$.

The next lemma is a consequence of Hayashi's connecting lemma in~\cite{Hay:97}.

\begin{lemma}[Persistence of heterodimensional cycles]
\label{l.hetcyc}
    Let $\mathcal{U} \subset \mathrm{Diff}^1(M)$ be an open set  such that 
    there are continuous families of saddles $(p_f)_{f\in \mathcal{U}}$ and $(q_f)_{f\in \mathcal{U}}$
       with different $\mathrm{s}$-indices  such that 
    $$
    H(p_f)=H(q_f) \qquad \mbox{for every $f \in \mathcal{U} \cap \mathcal{R}_0(M)$.}
    $$
    Then there exists a dense subset $\mathcal{D}$ of $\mathcal{U}$ such that every $g \in \mathcal{D}$ has a heterodimensional cycle associated to  $p_g$ and $q_g$.
\end{lemma}

 The orbit $\mathcal{O}(q)$ of 
 a saddle $q$  and a {\em{blender-horseshoe $B^+=(f, \mathbf{C}, \Gamma)$
are homoclinically related}} 
 if there is some saddle of $\Gamma$  whose orbit is homoclinically related to $\mathcal{O}(q)$.
In such a case, due to transitivity of $\Gamma$, the orbit of any saddle of $\Gamma$ is homoclinically related to $\mathcal{O}(q)$.
For the next result recall the family of disks $\mathfrak{D}_{\mathrm{bet}}(B^+)$ of a blender-horseshoe,
see Remark~\ref{r.defdisksinbet}.

\begin{prop}
\label{p.blender2}
    Consider $f \in \mathrm{Diff}^1(M)$ with a cycle of co-index $1$ with real central eigenvalues associated to saddles $p_f$ and $q_f$ such that 
    $$
\mbox{$\mathrm{s}$-$\mathrm{ind}(p_f)= \mathrm{s}$-$\mathrm{ind}(q_f)+1\quad$
   and $\quad H(p_f)=H(q_f)$.}
    $$
    Then there are $R>0$ and  an open set $\mathcal{V}_{p,f}$  of $\mathrm{Diff}^1(M)$, 
    with  $f \in \overline{\mathcal{V}_{p,f}}$,
   consisting of diffeomorphisms
  $g$ with an unstable blender-horseshoe $B^+_g$ homoclinically related to $\mathcal{O}(q_g)$ such that
        \begin{enumerate}
            \item every disk in $\mathfrak{D}_{\mathrm{bet}}(B^+_g)$
            intersects transversally 
            $W^\mathrm{s}_R(\mathcal{O}(p_g))$, and
            \item $W^\mathrm{u}(\mathcal{O}(p_g))$ contains a disk in $\mathfrak{D}_{\mathrm{bet}}(B^+_g)$.
            \item $W^\mathrm{s}(\mathcal{O}(p_g)) \pitchfork W^\mathrm{u}(B^+_g) \neq \emptyset \neq 
            W^\mathrm{u}(\mathcal{O}(p_g)) \cap W^\mathrm{s}(B^+_g).$
        \end{enumerate}
        \end{prop}

\begin{remark}
\label{r.itisenough}
By the $\lambda$-lemma, 
it is enough to prove Proposition~\ref{p.blender2} for any saddle $\bar p_f$ which is homoclinically related to
$p_f$.
\end{remark}

\begin{remark}
\label{r.blender2}
By symmetry, under the assumptions of Proposition~\ref{p.blender2}, there exist $R>0$ and an open set
$\mathcal{V}_{q,f}$, with $f \in \overline{\mathcal{V}_{q,f}}$, consisting of  diffeomorphisms  $g$
admitting a stable blender-horseshoe $B^-_g$ satisfying the analogous properties with $q_g$,
obtained by exchanging stable and unstable directions.
\end{remark}

We postpone the proof of Proposition~\ref{p.blender2} 
to Section~\ref{ss.proofofpropblender2}
and proceed to prove Theorem \ref{t.classeblender} assuming it.

\subsection{Proof of Theorem~\ref{t.classeblender}}
\label{ss.proofoftheoremclassblender}

Let $f \in \mathcal{R}_0(M)$, $H_f=H(p_f)=H(q_f)$, and $\mathcal{U}_f$  as in the hypotheses of  the theorem. 
 By Lemma~\ref{l.hetcyc}, there is a dense subset $\mathcal{D}$ of $\mathcal{U}_f$ such that every $h \in \mathcal{D}$ has a  cycle associated to $p_{h}$ and $q_{h}$.  Since,  by hypotheses, the homoclinic classes are 
 in  $\mathrm{ISPH}(h)$, these cycles have co-index one and have real central eigenvalues.
  By Proposition~\ref{p.blender2}, there is an open set $\mathcal{V}_{p,h}\subset \mathcal{U}_f$, with $h \in \overline{\mathcal{V}_{p,h}}$, such that every $g \in \mathcal{V}_{p,h}$ has an unstable blender-horseshoe $B^+_g$ satisfying 
 the conclusions of Proposition~\ref{p.blender2}.

Define the open and dense subset of $\mathcal{U}_f$,
    \begin{equation}
        \mathcal{A}_p \eqdef \bigcup_{h \in \mathcal{D}} \mathcal{V}_{p,h}.
    \end{equation}
Analogously, 
interchanging the roles of $p$ and $q$,
 given $h \in \mathcal{D}$, we get an open set  $\mathcal{V}_{q,h}\subset \mathcal{U}_f$, 
with $h \in \overline{\mathcal{V}_{q,h}}$, such that every $g \in \mathcal{V}_{q,h}$ has a stable blender-horseshoe $B^-_g$ satisfying conditions in Remark~\ref{r.blender2}.
   
    Define the open and dense subset of $\mathcal{U}_f$
    \begin{equation}
        \mathcal{A}_q \eqdef \bigcup_{h \in \mathcal{D}} \mathcal{V}_{q,h}.
    \end{equation}
Finally,  we consider the locally residual subset of $U_f$ 
    \begin{equation}
    \label{e.localresidual}
    \mathcal{R}_{\mathcal{U}_f} \eqdef \mathcal{A}_p \cap  \mathcal{A}_q \cap \mathcal{R}_0(M).
   \end{equation}

\begin{lemma}
\label{cl.unstableblendercontido}
The blender-horseshoes provided by Proposition~\ref{p.blender2} satisfy
  $$
  B_g^+,B^-_g \subset H(p_g)=H(q_g), \qquad \mbox{for every $g \in \mathcal{R}_{\mathcal{U}_f}$.}
  $$
\end{lemma}

\begin{proof}
    Take any $g \in  \mathcal{R}_{\mathcal{U}_f}$. It is enough to prove that $B^+_g \subset H(p_g)$.  
    The inclusion $B^-_g \subset H(q_g)$ is obtained 
    analogously.  
    Since  $g \in \mathcal{R}_{\mathcal{U}_f}$, by Condition~\ref{G1} in Theorem~\ref{t.residual}, $\mathrm{CR}(p_g)=H(p_g)=H(q_g)$. Hence, it is enough to prove $B^+_g \subset \mathrm{CR}(p_g)$.

    \begin{claim}
    \label{af.chainrecur}
        $B_g^+ \subset \mathrm{CR}(p_g)$.
    \end{claim}

    \begin{proof}
        Given any $\epsilon>0$, we prove that for every $z \in B^+_g$, there is an $\epsilon$-pseudo orbit starting at $p_g$ and ending at $z$. Analogously, we get an  $\epsilon$-pseudo orbit from $z$ to $p_g$. These two facts imply the claim.
        
        To construct the first
        $\epsilon$-pseudo, note that, by (3) in Proposition~\ref{p.blender2},  there is        
        $
            y \in W^{\mathrm{u}}(p_g) \cap W^{\mathrm{s}}(B^+_g).
        $
     Therefore, there are $n_1, n_2 \in \N$ and $w\in B^+_q$, such that 
            $$
            d(f^{-n_1}(y),p_g)<\frac{\epsilon}{2} \qquad \mbox{and} \qquad d(f^{n_2}(y),w)<\frac{\epsilon}{2}.
            $$ 
        Since $w,z \in B^+_g$ and $B_g^+$ is a transitive set, by Remark~\ref{r.classerecorrenciatransitivo}, there is an $(\epsilon/2)$-pseudo orbit $(x_i)_{i=0}^{n_3}$ starting at 
        $w$ and ending at $z$. 
        Concatenating  the point $p_g$ with the orbit segment $[f^{-n_1+1}(y),n_1+n_2-1]$ and  the $(\epsilon/2)$-pseudo orbit $(x_i)_{i=0}^{n_3}$, we obtain an $\epsilon$-pseudo orbit from $p_g$ to $z$.
        This ends the claim.
    \end{proof}
  Claim~\ref{af.chainrecur} implies that    $
        B^+_g \subset \mathrm{CR}(p_g)=H(p_g),$
proving the lemma.
\end{proof}

To finish the proof of the theorem, it remains to prove the occurrence of flip-flop configurations, this is obtained in the next lemma:

\begin{lemma}
    Let $g  \in \mathcal{R}_{\mathcal{U}_f}$. Then every saddle in $H(p_g)=H(q_g)$ with negative central Lyapunov exponent forms a split flip-flop configuration with $B^+_g$.
\end{lemma}
\begin{proof}

    By Condition \ref{G2} in Theorem~\ref{t.residual} and Lemma~\ref{l.flipflopeverypoint}, it suffices to verify that $p_g$ and $B_g^+$ form a split flip-flop configuration. We thus check the conditions of Definition~\ref{d.splitflipflop} for this pair. For notational simplicity, assume that $p_g$ is a fixed point. 
    Proposition~\ref{p.blender2} then yields the following claim, which establishes item~(1) in Definition~\ref{d.splitflipflop}.

    \begin{af}
        $W^\mathrm{u} (p_g)$ contains a disk in $\mathfrak{D}_{\mathrm{bet}}(B^+_g)$.
    \end{af}

Next claim provides  item~(2) in Definition~\ref{d.splitflipflop}.
    Consider small $\delta>0$ as in Remark~\ref{r.safetyregionofsaddles} and  the family 
    $\mathfrak{D}(p_g)=\mathfrak{D}_{\delta, \rho}(p_g)$ 
   in Remark~\ref{rd.familiadiscos}.

    \begin{af}
        There is $N_{p_g} \eqdef N_{p_g,\delta}>0$ such that $f^{N_{p_g}}(D)$ contains a disk in 
        $\mathfrak{D}(p_g)$ for every $D \in \mathfrak{D}_{\mathrm{bet}}(B^+_g)$.
    \end{af}
    \begin{proof}
    By Proposition~\ref{p.blender2}, there is $R>0$ such that 
           \begin{equation}
            D \pitchfork W_R^\mathrm{s} (p_g)\neq \emptyset, \qquad \mbox{for every $D \in
            \overline{\mathfrak{D}_{\mathrm{bet}}(B^+_g)}$.}
        \end{equation}
    Therefore, there exists $N_{p_g,D}>0$ such that $f^{n}(D)$ contains a disk in $\mathfrak{D}(p_g)$ for every $n \geq N_{p_g,D}$. Since $W^\mathrm{s}_R(p_g)$ is a compact subset of $W^\mathrm{s}(p_g)$, we can define
    \begin{equation}
        N_{p_g} \eqdef \sup_{D \in \overline{ \mathfrak{D}_{\mathrm{bet}}(B^+_g)}}\{N_{p_g,D}\}.
    \end{equation}
    By construction, every $f^{N_{p_g}}(D)$ contains a disk in $\mathfrak{D}(p_g)$ for every $D \in \mathfrak{D}_{\mathrm{bet}}(B^+_g)$, ending the proof of the claim.
    \end{proof}
   We have proved that Items (1) and  (2) of Definition~\ref{d.splitflipflop} hold for $p_g$ and $B^+_g$, 
   ending the proof of the lemma.
   \end{proof}
   The proof of the theorem is now complete.
\hfill \qed

\subsection{Proof of Proposition~\ref{p.blender2}}
\label{ss.proofofpropblender2}
This proof is based on the refinement given in \cite{BonDiaKir:12} of the arguments in \cite{BonDia:08} yielding the \emph{stabilization of heterodimensional cycles}, which we now describe.

For what follows, recall that every hyperbolic basic set $\Lambda_f$ of $f\in \mathrm{Diff}^1(M)$ has a well-defined continuation $\Lambda_g$ for every $g$ sufficiently close to $f$. The $\mathrm{s}$-index of a basic set is defined as the  $\mathrm{s}$-index of any of its saddles.

Similarly to the case of cycles associated to saddles,
two hyperbolic basic  sets $\Sigma_f$ and $\Gamma_f$ of a diffeomorphism $f$ form a  
{\em{(heterodimensional) cycle}} 
 if
    \begin{itemize}
         \item the $\mathrm{s}$-indices of $\Sigma_f$ and $\Gamma_f$ are different, and
        \item $W^{\mathrm{s}}(\Sigma_f) \cap W^{\mathrm{u}}(\Gamma_f) \neq \emptyset \neq W^{\mathrm{u}}(\Sigma_f) \cap W^{\mathrm{s}}(\Gamma_f)$.
    \end{itemize}   
This cycle has {\em{co-index one}} if the $\mathrm{s}$-indices of $\Sigma_f$ and $\Gamma_f$ differ by one. 
The cycle is {\em{robust}} if there is a $C^1$-neighbourhood of $f$ consisting of diffeomorphisms $g$ having a cycle associated to the continuations $\Sigma_g$ and $\Gamma_g$. 

A cycle associated to a pair of saddles $p_f$ and $q_f$ of $f$ is $C^1$-stabilizable if there is a
$C^1$-open set $\mathcal{U}_f$ with $f\in \overline{\mathcal{U}(f)}$ such that every $g\in \mathcal{U}_f$
has a robust cycle associated to basic sets $\Sigma_g$ and $\Upsilon_g$ containing the continuations
$p_g$ and $q_g$.

Since the cycles of the diffeomorphisms in the set
$\mathcal{D}$ in Lemma~\ref{l.hetcyc}
are associated with saddles with nontrivial homoclinic classes, the 
theorem below can be applied to them.

\begin{teo}[Theorem A in \cite{BonDiaKir:12}]
\label{t.bondiakir}
Every co-index one cycle associated with a pair of saddles, one of which has a nontrivial homoclinic class, is $C^1$-stabilizable.
\end{teo}

\begin{remark}\label{r.howarethecycles}
The cycles in Theorem~\ref{t.bondiakir}
are associated with the continuation of one of the saddles in the initial cycle and with a blender-horseshoe containing the continuation of the other saddle.
\end{remark}

We now review the proof Theorem~\ref{t.bondiakir} to deduce  Proposition~\ref{p.blender2}. 
We begin by introducing {\em{strong homoclinic intersections}}
associated with {\em{neutral points}} and recall that they yield blender-horseshoes. 

In the strongly partially 
hyperbolic setting, a
 periodic point $s$ of $f$ is {\em{neutral}}  if the derivative 
$Df^{\pi(s)}(s)$ has one eigenvalue equal to $\pm 1$.
These points are either of {\em{saddle-node}} or {\em{flip}} type (corresponding to the cases $+1$ 
and $-1$, respectively)\footnote{This means that the point is partially hyperbolic, with the terminology
in \cite{BonDiaKir:12}}. 
Note that 
$$
    T_{\mathcal{O}(s)}M = E^\s_{\mathcal{O}(s)} \oplus E^{\mathrm{c}}_{\mathcal{O}(s)} \oplus E^\u_{\mathcal{O}(s)},
$$
where  
$E^{\mathrm{c}}_{\mathcal{O}(s)}$ 
corresponds to the eigenvalue $\pm 1$. This allows 
to consider the strong stable $W^\s (\mathcal{O}(s) )$ and strong unstable $W^\u(\mathcal{O}(s))$ manifolds of
$\mathcal{O}(s)$. 
We say that  $\mathcal{O}(s)$ has a  {\em{strong homoclinic intersection}} if $W^\s(\mathcal{O}(s)) \cap W^\u(\mathcal{O}(s))$ 
contains a point $w\not\in \mathcal{O}(p)$, called a {\em{strong homoclinic point}},  such that 
\begin{equation}
    T_w W^\s(\mathcal{O}(s)) + T_w W^\u(\mathcal{O}(s)) = T_wW^\s(\mathcal{O}(s)) \oplus T_w W^\u(\mathcal{O}(s)).
\end{equation}

By \cite[Theorem 4.1]{BonDia:08}, these intersections yield robust  cycles associated 
with a blender-horseshoe\footnote{The terminology \emph{blender-horseshoe} was introduced in the subsequent work~\cite{BonDia:12}, where the hyperbolic sets considered in~\cite{BonDia:08} were identified as blender-horseshoes.}
 and a saddle. To fix ideas and be more precise, consider  the saddle-node case. 
On the one hand, the saddle-node splits into two saddles \(s^+\) and \(s^-\) with different \(\mathrm{s}\)-indices, the one of the former being greater.
On the other hand, the unfolding of
the intersection between the strong stable and unstable manifolds generates a unstable blender-horseshoe containing  $s^+$ (this is precisely stated in \cite[Proposition 5.4]{BonDia:12}).
The robust cycle is associated to this blender and the saddle $s^-$. 
 One can reverse the roles of the saddles, getting a robust cycle involving a stable blender-horseshoe and the 
 saddle $s^+$.
 
 Indeed, what is implicitly proved in \cite{BonDia:12} is that one can select one of the saddles in the cycle and obtain
a $C^1$-robust cycle involving that saddle and a blender-horseshoe containing the other one. 
This construction was made explicit in \cite{BonDiaKir:12}.
Let us sketch it. For this let us recall that the 
cycles  of co-index one associated with saddles with real eigenvalues are classified in 
\cite{BonDiaKir:12} as {\em{twisted}} and {\em{nontwisted}} (in \cite{LiTur:24} these cycles are renamed
type I and II, respectively).

\begin{lemma} \label{l.twistedafter}
 Every co-index one cycle associated to  saddles $p_f$ and $q_f$
with nontrivial homoclinic classes yields after an arbitrarily small $C^1$-perturbation a nontwisted cycle associated to saddles 
$\bar p_g$ and $\bar q_g$ which are homoclinically related to $p_g$ and $q_g$.
\end{lemma}

\begin{proof}
 The lemma  is a consequence of the following results:
\cite[Lemma 5.9]{BonDiaKir:12} (abundance of saddles with the bi-accumulation property
in nontrivial homoclinic classes and cycles) and
\cite[Proposition 6.2]{BonDiaKir:12} (twisted cycles associated to saddles with the bi-accumulation
property yields nontwisted cycles).
\end{proof}

The proposition now follows from the following result:

\begin{prop}[Proposition 6.1 in \cite{BonDiaKir:12}]
\label{p.bondiakirmany}
Let $f\in \mathrm{Dif}^1(M)$ have a nontwisted cycle associated with saddles $\bar p_f$ and $\bar q_f$
with $\mathrm{s}$-index$(\bar p_f)= \mathrm{s}$-index$(\bar q_f)+1$. Then there is
$g\in \mathrm{Dif}^1(M)$ arbitrarily $C^1$-close to $f$ with a neutral point $s_g$ such that
\begin{enumerate}
\item[{\rm{(a)}}]
$W^\s (\mathcal{O}(s_g)) \pitchfork W^\mathrm{u}(\mathcal{O}(\bar q_g))\ne\emptyset$;
\item[{\rm{(b)}}]
$W^\u (\mathcal{O}(s_g)) \pitchfork W^\mathrm{s}(\mathcal{O}(\bar p_g))\ne\emptyset$; 
\item[{\rm{(c)}}]
a strong homoclinic intersection
 $W^\u (\mathcal{O}(s_g) )\cap W^\u \mathcal{O}((s_g))$;
 \item[{\rm{(d)}}]
 $W^\s (\mathcal{O}(s_g)) \cap W^\mathrm{u}(\mathcal{O}(\bar p_g))\ne\emptyset$;
and
\item[{\rm{(e)}}]
$W^\u (\mathcal{O}(s_g)) \cap W^\mathrm{s}(\mathcal{O}(\bar q_g))\ne\emptyset$.
\end{enumerate}
\end{prop}

Using
Proposition~\ref{p.bondiakirmany}, we get $h\in \mathrm{Dif}^1(M)$ arbitrarily $C^1$-close to $g$
 having a
unstable blender-horseshoe $B_h^+$ such that 
the conclusions (1)--(3) of Proposition~\ref{p.blender2} hold for 
$B_h^+$ and $\bar p_h$ as follows:
\begin{enumerate}
\item[{\rm{(i)}}]
The strong homoclinic intersection in (c) provides a unstable blender-horseshoe $B^+_h$.
Item (b) implies that $W^\mathrm{u} (B^+_h) \pitchfork W^\mathrm{s} (\mathcal{O}(\bar p_h)\ne\emptyset$.
Since the disks in $\mathfrak{D}_{\mathrm{bet}}(B^+_g)$ are close
to the (local) strong unstable set of $s_h$, item (b) also implies that
every disk in $\mathfrak{D}_{\mathrm{bet}}(B^+_h)$
intersects $W^\mathrm{s}_R (\mathcal{O}(\bar p_h)$ for some $R$,
proving item (1) of the proposition.
   \item[{\rm{(ii)}}] 
Unfolding the intersection in (d), we get that 
$W^\mathrm{u}(\mathcal{O}(\bar p_h)$ contains a disk 
in  $\mathfrak{D}_{\mathrm{bet}}(B^+_h)$, proving item (2).
\item[{\rm{(iii)}}]
the intersections
 in item (3) follow from the two previous items:
 $W^\mathrm{u} (B^+_h) \pitchfork  W^\mathrm{s}(\mathcal{O}(\bar p_h)$ is explicit in (i) and
  $W^\mathrm{s} (B^+_h) \cap  W^\mathrm{u}(\mathcal{O}(\bar p_h)\ne \emptyset$ follows from (ii) and  
  Remark~\ref{r.intersectionwiththestableset}.
\end{enumerate}

Note that these three conditions are $C^1$-open ones. Thus for each $g$ as in the theorem
we get an open set $\mathcal{U}_g$ formed by diffeomorphisms $h$ as above. As $g$ can be chosen arbitrarily
close to $f$, considering the union of these sets we get a set $\mathcal{V}_{p,f}$ whose closure contains $f$.
As the saddles $\bar p_h$ are homoclinically related to $p_h$ the result holds for $p_h$, recall Remark~\ref{r.itisenough}.
This completes the proof of the proposition. \qed

\section{Proof of Theorem~\ref{t.teogeralclassehom}: Approximation of nonhyperbolic measures}
\label{s.thmAnonhyperbolic}

In this section, we define the residual set from Theorem~\ref{t.teogeralclassehom} and show that its diffeomorphisms satisfy the hypotheses of Theorem~\ref{t.densityisolatedset}, see
Section~\ref{ss.endoftheproofofA}. This yields Theorem~\ref{t.teogeralclassehom} 
for $\alpha=0$. For this end, we state the approximation result Proposition~\ref{p.apnonhyphomclassperhyp}
below.

\subsection{Approximation of nonhyperbolic measures by periodic ones}
\label{ss.insideapproximation}

Given $f \in \mathrm{Diff}^1(M)$ and a homoclinic class $H\in \mathrm{ISPH}(f)$ (see Definition~\ref{d.ISPH}),
 recall the three types of measures
 $\mathcal{M}^\dag_{\mathrm{erg}}(f,H)$, $\dag\in \{0,+,-\}$,
 in \eqref{e.threemeasures}.

\begin{prop}
\label{p.apnonhyphomclassperhyp}
    There is a residual subset $\mathcal{R}_1(M)$ of $\mathrm{Diff}^1(M)$ consisting of diffeomorphisms
   $f$ with the following property: for every homoclinic class $H$ in $\mathrm{ISPH}(f)$ 
   and   every $\mu \in \mathcal{M}^0_{\mathrm{erg}}(f,H)$, there exist sequences of periodic measures 
     $(\mu^+_k)_k$ in  $\mathcal{M}^+_{\mathrm{erg}}(f,H)$ 
      and $(\mu^-_k)_k$  in $\mathcal{M}^-_{\mathrm{erg}}(f,H)$
     that converge to $\mu$.
\end{prop}

\begin{proof}
    To prove the proposition we use Theorem~\ref{t.abc} below.

\begin{teo}[Particular case of Theorem 3.5 in~\cite{AbdBonCro:11}]
\label{t.abc}
    There is a residual subset $\mathcal{R}_2(M)$ of $\mathrm{Diff}^1(M)$ such that for every $f \in \mathcal{R}_2(M)$ and every  $H$ in $\mathrm{ISPH}(f)$  
      the set of periodic measures supported on $H$ is dense in $\mathcal{M}_{\mathrm{inv}}(f,H)$.
\end{teo}

Recall the residual  set $\mathcal{R}_0(M)$ of $\mathrm{Diff}^1(M)$ in  Theorem~\ref{t.residual} and 
consider the new residual subset $\mathcal{R}_1(M)$ of $\mathrm{Diff}^1(M)$ defined by
        \begin{equation}
        \label{e.defR11}
            \mathcal{R}_1(M) \eqdef \mathcal{R}_2(M) \cap \mathcal{R}_0(M).
        \end{equation}
Let $f \in \mathcal{R}_1(M)$ and $H\in \mathrm{ISPH}(f)$. By condition (G6)
there is  $\mu \in \mathcal{M}^0_{\mathrm{erg}}(f,H) \subset \mathcal{M}_{\mathrm{inv}}(f,H)$. By Theorem~\ref{t.abc}, there is a sequence $(\mu_k)_k$ of periodic measures in $\mathcal{M}_{\mathrm{inv}}(f,H)$ converging to $\mu$. By Condition~\ref{G0}, these measures are hyperbolic, thus  $\mu_k \in \mathcal{M}^-_{\mathrm{erg}}(f,H) \cup \mathcal{M}^+_{\mathrm{erg}}(f,H)$. After taking a subsequence, we can suppose, for example, that $\mu_k \in \mathcal{M}^-_{\mathrm{erg}}(f,H)$ for every $k \in \N$. We let $\mu_k^- \eqdef \mu_k$. 
    
It remains construct the sequence $(\mu_k^+)_k$. For that, consider a saddle $q \in H$ with $\chi^{\c}(q)>0$ provided by the index variation condition of $H$ 
and the periodic measure
$\mu_q=\mu_{\mathcal{O}(q)}$.  
Define 
        \begin{equation}
        \label{eq.escolhalambdak}
            \lambda_k \eqdef \frac{1}{k(\chi^{\mathrm{c}}(\mu^-_k)-\chi^{\mathrm{c}}(\mu_q))} - \frac{\chi^{\mathrm{c}}(\mu_q)}{\chi^{\mathrm{c}}(\mu^-_k)-\chi^{\mathrm{c}}(\mu_q)}, \qquad k \in \N. 
        \end{equation}
      Since $\chi^{\mathrm{c}}(\mu^-_k)<0$, $\chi^{\mathrm{c}}(\mu^-_k) \to 0$, and 
      $\chi^{\mathrm{c}}(\mu_q)>0$, it immediately follows:
    \begin{af}
    \label{af.escolhalmabdak}
        For every $k$ sufficiently large, $0< \lambda_k <1$ and $\lim_{k \to \infty} \lambda_k=1$.
    \end{af}

   Define the nonergodic invariant measures supported on $H$,
    \begin{equation}
            \nu_k \eqdef (1-\lambda_k)\mu_q + \lambda_k \mu_k^-.
        \end{equation}

    The proposition now follows from the next claim.
        
    \begin{af}
        The sequence $(\nu_k)_k$ 
            converges to $\mu$ and satisfies $\le^{\mathrm{c}}(\nu_k)>0$ for every $k \in \N$. 
    \end{af}
    \begin{proof}
  For the convergence, take  any continuous map $\varphi\colon M \to \mathbb{R}$ and note that
            $$                \int \varphi d \nu_k   = \int \varphi d((1- \lambda_k) \mu_q + \lambda_k  \mu_k^-) 
                 = (1-\lambda_k)\int \varphi d\mu_q + \lambda_k \int \varphi d\mu^-_k. 
           $$
                   Since $\mu^-_k\to \mu$ and $\lim_{k \to \infty} \lambda_k=1$, the convergence to $\mu$ follows noting that 
            \begin{equation}
                \int \varphi d \nu_k \to \int \varphi d\mu, \quad \mbox{as $k \to \infty$}.
            \end{equation}
       
   To see that $\le^{\mathrm{c}}(\nu_k)>0$
  for every $k$  as in Claim~\ref{af.escolhalmabdak}, write 
            \[
            \begin{split}
                \le^{\mathrm{c}}(\nu_k)  &= (1-\lambda_k) \le^{\mathrm{c}}(\mu_q) + \lambda_k \le^{\mathrm{c}}(\mu^-_k) = \le^{\mathrm{c}}(\mu_q) + \lambda_k(\le^{\mathrm{c}}(\mu^-_k)-\le^{\mathrm{c}}(\mu_q)) =\\
            { \footnotesize{\mbox{($\lambda_k$ as in~\eqref{eq.escolhalambdak})}}}         & = \le^{\mathrm{c}}(\mu_q)  
                 +\frac{\le^{\mathrm{c}}(\mu^-_k)-\le^{\mathrm{c}}(\mu_q)}{k(\le^{\mathrm{c}}(\mu^-_k)-\le^{\mathrm{c}}(\mu_q))}
                - \frac{\le^{\mathrm{c}}(\mu_q) (\le^{\mathrm{c}}(\mu^-_k)-\le^{\mathrm{c}}(\mu_q))}{\le^{\mathrm{c}}(\mu^-_k)-\le^{\mathrm{c}}(\mu_q)} = \frac{1}{k}>0,
            \end{split}
            \]
       proving the assertion and ending  the proof of claim.
    \end{proof}
    The proof the proposition is now complete. \qed
  
  \subsection{End of the proof of~Theorem~\ref{t.teogeralclassehom} (nonhyperbolic case)}
  \label{ss.endoftheproofofA}
  Define the residual set of $\mathrm{Dif}^1(M)$ as follows:
  $$
  \mathcal{R}(M)=\mathcal{R}_1(M) \cap \mathcal{R}_3 (M),
  \qquad 
  \mathcal{R}_3 (M) \eqdef \bigcup_{f\in \mathcal{R}_0} \mathcal{R}_{\mathcal{U}_f},
  \quad \mbox{as in Theorem~\ref{t.classeblender}.}
  $$
 By construction, given any $f\in  \mathcal{R}(M)$,
 every homoclinic class in $\mathrm{ISPH}(f)$ satisfies the hypotheses of Theorem~\ref{t.densityisolatedset}:
 item (1) follows from Theorem~\ref{t.classeblender} and
 item (2)
  follows from $f\in \mathcal{R}_1(M)$ and Proposition~\ref{p.apnonhyphomclassperhyp}.
The proof of the theorem is now complete.  
\end{proof}

\section{Proof of Theorem~\ref{t.teogeralclassehom}:
Approximation of hyperbolic measures} \label{s.hyperbolic}

In this section, we prove Theorem~\ref{t.teogeralclassehom} for exponents $\alpha \neq 0$.  
In fact, we prove a slightly stronger statement, Theorem~\ref{t.hyperboliccase}, in which no index variability is required. For this result, recall the residual set $\mathcal{R}_0(M)$ from Theorem~\ref{t.residual} and the numbers $\alpha_{\inf}(H)$ and $\alpha_{\sup}(H)$ defined in~\eqref{eq.defalphaminalphamax}.

\begin{teo}
\label{t.hyperboliccase}
Let $f \in \mathcal{R}_0(M)$ and $H$ be a homoclinic class of $f$ that is nontrivial, strongly partially hyperbolic, and isolated. Then, for every  
\[
\alpha \in \bigl(\alpha_{\inf}(H),\alpha_{\sup}(H)\bigr) \setminus \{0\},
\]  
the set $\mathcal{M}^{\alpha}_{\mathrm{aGIKN}}(f,H)$ is nonempty and dense in $\mathcal{M}^{\alpha}_{\mathrm{erg}}(f,H)$.    
\end{teo}

Note that, since the periodic points of any $f \in \mathcal{R}_0(M)$ are hyperbolic, every nonhyperbolic homoclinic class $H$ for such $f$ is automatically nontrivial.

The proof of this theorem closely follows that of Theorem~\ref{t.densityisolatedset}, so we omit some details. There are two main differences. First, blender-horseshoes are not needed, because we can ``interpolate exponents''  between saddles of the same $\mathrm{s}$-index, which are homoclinically related. Second, we need to carefully adjust the convergence of the central exponents so that they approach the exponent of the limiting measure. Unlike the previous situation, this convergence cannot be assumed exponential (compare Claim~\ref{cl.semlabel} with Claim~\ref{cl.seconsemnome}).

The proof is organized as follows. In Lemma~\ref{l.apmethodgeral}, we construct periodic measures using Lemma~\ref{l.gan}, in the spirit of Lemma~\ref{l.periodicpointshomoclinicclass}. Finally, by applying Lemma~\ref{l.apmethodgeral} inductively, we obtain the hyperbolic GIKN measures described in Lemma~\ref{l.explyphyp}.

We start with some preliminaries.
Note that since hyperbolic measures are approximated by periodic ones and 
every pair of saddles of the same $\mathrm{s}$-index are homoclinically related, see 
Theorem~\ref{t.abc} and
condition
(G2) in Theorem~\ref{t.residual}, 
we have that $\mathcal{M}_{\mathrm{erg}}^\alpha(f,H)$ is nonempty for every $\alpha \in (\alpha_{\inf}(H),\alpha_{\sup}(H))$.

There are two cases to consider in the theorem: $0 \in (\alpha_{\inf}(H), \alpha_{\sup}(H))$ and
 $0 \notin (\alpha_{\inf}(H), \alpha_{\sup}(H))$. We consider the first case and assume that 
 $\alpha \in (\alpha_{\inf}(H), 0)$. The other cases are analogous and thus omitted.
 Theorem~\ref{t.hyperboliccase} follows from the following 
 Proposition~\ref{pt.densityhomclass2}:

\begin{prop}
\label{pt.densityhomclass2}
    Let $f $ and 
    $H$ be as in the hypotheses of Theorem~\ref{t.hyperboliccase}. Assume  that
   $\alpha_{\mathrm{inf}}(H)<0<  \alpha_{\mathrm{sup}}(H)$.
       Then, for every $\alpha \in (\alpha_{\inf},0)$, 
       the set $\mathcal{M}_{\mathrm{aGIKN}}^\alpha(f,H)$ is 
       dense in $\mathcal{M}_{\mathrm{erg}}^\alpha(f,H)$. 
     \end{prop}

\subsection{Proof of Proposition \ref{pt.densityhomclass2}}

    We begin with a lemma whose proof is postponed to Section~\ref{ss.proofoflemmaapmethodgeral}.

\begin{lemma}
\label{l.apmethodgeral} 
    Consider numbers
    $\epsilon,\delta,\tau>0$  and periodic orbits $\mathcal{O}_1$ and $\mathcal{O}_2$ 
    of $f$ contained in $H$
    with negative central Lyapunov exponents $\chi^{\mathrm{c}}_1 =  \chi^{\mathrm{c}} (\mathcal{O}_1)$ 
    and $\chi^{\mathrm{c}}_2=  \chi^{\mathrm{c}} (\mathcal{O}_2)$.
    Then for every  $\rho \in (\tau, 1)$, there exists a periodic orbit of $f$ in b$H$
    $$
    \widetilde{\mathcal{O}}=
    \widetilde{\mathcal{O}}_{(\mathcal{O}_1,\mathcal{O}_2,  \epsilon, \delta, \rho)},
    $$
   which is an $(\epsilon,\rho-\tau)$-good approximation of $\mathcal{O}_1$
    whose central exponent satisfies
       $$     
    \chi^{\mathrm{c}}( \widetilde{\mathcal{O}} ) \in [ \rho \chi^{\mathrm{c}}_1 + (1-\rho)\chi^{\mathrm{c}}_2 - \delta \, , \, \rho \chi^{\mathrm{c}}_1 + (1-\rho)\chi^{\mathrm{c}}_2 + \delta  ].
    $$ 
    \end{lemma}

    Next lemma states the existence of hyperbolic GIKN measures and it is the corresponding version of Corollary~\ref{c.periodicpointshomclass} in the hyperbolic case.

    \begin{lemma}\label{l.explyphyp}
    Let $\mathcal{O}_1$ be a periodic orbit of $f$ contained in $H$ with
$\chi^{\mathrm{c}}_1 = \chi^{\mathrm{c}}(\mathcal{O}_1) \in (\alpha_{\inf}(H),0)$.
Fix $\alpha_\ast \in (\chi^{\mathrm{c}}_1,0)$.
Then there exist a sequence of positive numbers $(\gamma_n)_{n \in \N}$ converging to $0$
and a GIKN sequence of periodic measures $(\nu_n)_{n \in \N}$ such that their central exponents
$\chi^\mathrm{c}_{n}$ satisfy
\[
0
<
\left(
\frac{\chi^{\mathrm{c}}_1}{2^n}
+
\sum_{i=1}^n \frac{\alpha_\ast}{2^i}
\right)
-
\chi^{\mathrm{c}}_{n+1}
<
\gamma_n,
\qquad \text{for every } n \in \N,
\]
and whose limit is an aperiodic GIKN measure $\nu$ with
$\chi^\mathrm{c}(\nu) = \alpha_\ast$.
  \end{lemma}
\begin{proof}
The fact that $\le^\c (\nu)=  \alpha_\ast$ follows from $\le^\c (\nu_n) \to \le^\c (\nu)$, recall Remark~\ref{r.convexp}.

    We construct the sequence of periodic orbits 
using Lemma~\ref{l.apmethodgeral}. 
We start by choosing some quantifiers. Let
$
d_1 \eqdef d_{\mathcal{O}_1}
$, recall ~\eqref{e.distorbitaseq},
and define 
\begin{equation} 
\label{e.rho1}
\begin{split}
\epsilon_{1}  \eqdef  
\frac{d_1\cdot |\chi^{\mathrm{c}}_1- \alpha_{\ast}(H)|}{3 \cdot |\alpha_{\inf}|} \quad \mbox{and} \quad
\rho_1  \eqdef \frac{1}{\chi^{\mathrm{c}}_1}\left(\frac{\chi^{\mathrm{c}}_1}{2}+\frac{\alpha_{\ast}}{2}\right)\in (0,1).
\end{split}
\end{equation}
Take $\delta>0$ such that 
\begin{equation}
\label{e.delta1}
\delta_1 = \frac{\delta}{4} (1- \rho_1)^{-1} >0.
\end{equation}

\begin{claim}
\label{cl.indstep1}
    There is a periodic orbit $\mathcal{O}_2=\mathcal{O}_{(\mathcal{O}_1,\epsilon_1,\rho_1, \delta)}$ with 
    exponent $\chi^{\mathrm{c}}_2= \chi^{\mathrm{c}} (\mathcal{O}_2)$ such that
        $$
            0 < \left( \frac{\chi^{\mathrm{c}}_1}{2}+\frac{\alpha_\ast}{2} \right) -\chi^{\mathrm{c}}_2 < \frac{\delta}{2}.
        $$
\end{claim}
\begin{proof}
Since $\mathcal{M}_{\mathrm{erg}}^0(H)$ is nonempty,  Proposition~\ref{p.apnonhyphomclassperhyp}
provides a periodic orbit $\widetilde{\mathcal{O}}_1$ such that
 $\widetilde{\chi}^{\mathrm{c}}_1 =  \chi^{\mathrm{c}}(\widetilde{\mathcal{O}}_1)$   
satisfies
$
0<-\widetilde{\chi}^{\mathrm{c}}_1< \delta_1.
$
By the choice of $\rho_1$ in \eqref{e.rho1},
\begin{equation}
\label{e.lyapunoves}
\begin{split}
\frac{{\chi}^{\mathrm{c}}_1}{2}+\frac{\alpha_{\ast}}{2}-\left(\rho_1 \chi^{\mathrm{c}}_1 + (1-\rho_1) \widetilde{\chi}^{\mathrm{c}}_1\right) 
& = -(1-\rho_1) \widetilde{\chi}^{\mathrm{c}}_1 = (1-\rho_1) (-\widetilde{\chi}^{\mathrm{c}}_1)\\
&< (1-\rho_1)\delta_1  = (1-\rho_1)\frac{\delta}{4}(1-\rho_1)^{-1} = \frac{\delta}{4}.
\end{split}
\end{equation}
By Lemma \ref{l.apmethodgeral}, there is 
an orbit $\mathcal{O}_2 = \mathcal{O}_{(\mathcal{O}_1, \widetilde{\mathcal{O}}_1)}$ 
such that
$$
\chi^{\mathrm{c}}_2  = \chi^{\mathrm{c}} (\mathcal{O}_2) \in \left[\rho_1 \chi^{\mathrm{c}}_1 + (1-\rho_1) \widetilde{\chi}^{\mathrm{c}}_1- \frac{\delta}{4},\rho_1 \chi^{\mathrm{c}}_1 + (1-\rho_1) \widetilde{\chi}^{\mathrm{c}}_1 + \frac{\delta}{4}\right].
$$
The claim now follows using~\eqref{e.lyapunoves}.
\end{proof}

We proceed inductively, 
suppose that, for every $1 \leq i < n$, there are  periodic orbits $\mathcal{O}_i$ and $\widetilde{\mathcal{O}}_i$, with 
$\chi^{\mathrm{c}}_i= \chi^{\mathrm{c}} (\mathcal{O}_i)<0$ and 
$\widetilde{\chi}^{\mathrm{c}}_i = \chi^{\mathrm{c}} (\widetilde{\mathcal{O}}_i)<0$,  
and numbers 
\begin{equation}
\label{eq.rhodeltaepsilonn}
\begin{split}
d_i &\eqdef d_{\mathcal{O}_i},
\\
\rho_i &\eqdef \frac{1}{\chi^{\mathrm{c}}_i}\left(\frac{\chi^{\mathrm{c}}_1}{2^i}+\frac{(2^i-1)\alpha_{\ast}}{2^i}\right) \in (0,1), 
\\ 
\delta_i &\eqdef \frac{\delta}{2^{i+1}}(1-\rho_i)^{-1}>0, \quad \mbox{and}\\
\epsilon_{i} &\eqdef \frac{(\min_{j\leq i}d_j) \cdot |\chi^{\mathrm{c}}_1-\alpha_{\ast}|}{2^{i-1}\cdot 3 \cdot |\alpha_{\inf}(H)|} 
\leq \frac{d_i\cdot |\chi^{\mathrm{c}}_1-\alpha_{\ast}|}{2^{i-1}\cdot 3 \cdot |\alpha_{\inf}(H)|}
\end{split}
\end{equation}
such that
\begin{itemize}
\item
$\mathcal{O}_{i+1}=\mathcal{O}_{(\mathcal{O}_{i}, \widetilde{\mathcal{O}}_{i}, \epsilon_{i}, \rho_{i}, \delta_i)}$
for every $i+1\leq n$,
\item
$0<-\widetilde{\chi}^{\mathrm{c}}_i< \delta_i,$
\item
$
    0 
    < 
    \left(
    \dfrac{\chi^{\mathrm{c}}_1}{2^{i}} 
    +
    \sum_{j=1}^{i} \dfrac{\alpha_{\ast}}{2^j}
    \right)
    -
     \chi^{\mathrm{c}}_{i+1}
    <\dfrac{\delta}{2^i}.
$
\end{itemize}

The next claim provides the orbit $\mathcal{O}_{n+1}$.

\begin{claim}
\label{cl.seconsemnome}
    The periodic orbit $\mathcal{O}_{n+1}=\mathcal{O}_{(\mathcal{O}_{n},\widetilde{\mathcal{O}}_{n},\epsilon_n,\rho_n, \delta_n)}$ has central exponent $\chi^{\mathrm{c}}_{n+1}$ satisfying
        $$
            0 < \left( \frac{\chi^{\mathrm{c}}_1}{2^n}+ \sum_{i=1}^{n}\frac{\alpha_\ast}{2^i} \right) -\chi^{\mathrm{c}}_{n+1} < \frac{\delta}{2^n}.
        $$
\end{claim}
\begin{proof}
The proof follows as  the one of Claim~\ref{cl.indstep1}  noting that
 there is  a periodic orbit $\widetilde{\mathcal{O}}_{n}$ whose exponent $\widetilde{\chi}^{\mathrm{c}}_{n}$ satisfies
$0<- \widetilde{\chi}^{\mathrm{c}_n}< \delta_n.$
We will omit the details.
\end{proof}

\begin{remark}
    By construction, every orbit $\mathcal{O}_{n+1}$ is an $(\epsilon_n,\rho_n)$-good approximation of $\mathcal{O}_n$. Let $\nu_n\eqdef \mu_{\mathcal{O}_n}$ and $\chi_n^c$ its  exponent. Set $\gamma_n \eqdef \frac{\delta }{2^n}$. One has that,  
        \begin{equation}
        \label{eq.gamman}
            0  < \left(\dfrac{\chi^{\mathrm{c}}_1}{2^n}   + \sum_{i=1}^n \dfrac{\alpha_{\ast}}{2^i} \right)  -\chi^{\mathrm{c}}_{n+1}
            <\gamma_n.
        \end{equation}
\end{remark}

\begin{claim}
\label{cl.isGIKN}
The sequence $(\nu_n)_n$ is GIKN and converges to an 
 aperiodic measure $\nu$. 
\end{claim}

\begin{proof}
The fact below implies  that  $(\nu_n)_n$ is GIKN.

\begin{fact}
\label{fcl.sumepsilon}
    $\sum_{n}^{\infty} \epsilon_n < \infty$ and $\prod_{n=1}^{\infty} \rho_n >0.$ 
\end{fact}

\begin{proof}
The first claim follows from the definition of $\epsilon_n$ in~\eqref{eq.rhodeltaepsilonn} and $d_n < \diam(M)$.

For the second assertion, fix  $n \in \N$. By the definition of $\rho_n$ in~\eqref{eq.rhodeltaepsilonn} it follows 
$\lim_{n\to \infty} \chi^{\mathrm{c}}_n=\alpha_{\ast} \neq 0$. Hence, $\lim_{n \to \infty} \rho_n =1$. Therefore, given small $\eta>0$, there exists $n_0$ such that for every $n\geq n_0$ one has that
    \begin{equation}
    \label{eq.convto1}
        |\rho_n-1|<\frac{\eta}{2^n}.
    \end{equation}

    To conclude the proof of the assertion,    fix any $m=n+i >n \geq n_0$.
    By induction on $i$ and using~\eqref{eq.convto1} we get
    \begin{equation*}
    \label{eq.boundbyak-1}
        \left|\prod_{k=n}^{n+i} \rho_k-1\right| \leq \sum_{k=n}^{n+i} |\rho_k-1| < \sum_{k=n}^{n+i} \frac{\eta}{2^k} <\sum_{k=1}^{n+i} \frac{\eta}{2^k} <\eta,
    \end{equation*}
 ending the proof of the assertion.
 \end{proof}

Finally,
using~\eqref{eq.rhodeltaepsilonn} and that $\alpha_{\inf} < \chi^\mathrm{c}_1 < \alpha_\ast<0$ it follows
$$
        \sum_{k=n}^{\infty} \epsilon_k < d_n/3 
        \qquad \mbox{for every $n$.}
 $$
 Theorem~\ref{t.l.p.strongnonper} implies the aperiodicity of the limit measure $\nu$,
ending the proof of the claim.
 \end{proof}

The  proof of the  lemma is now complete. 
\end{proof}

We will use Lemma~\ref{l.p.cauchyhyp} below, which is the hyperbolic analogue of Proposition~\ref{p.mainprop}.
Its proof follows the same line of calculations as in the proof of Proposition~\ref{p.mainprop} and hence is omitted.

\begin{lemma}\label{l.p.cauchyhyp} 
There is a constant $A_{H}>0$ such that for every periodic measure $\nu \in \mathcal{M}_{\mathrm{erg}}(f,H)$  with $\chi^{\mathrm{c}} (\nu) <0$ and every $\alpha\in (\alpha_{\inf}(H), 0)$, there is $\mu \in \mathcal{M}_{\mathrm{aGIKN}}^\alpha(f,H)$  such that
$$
W_1(\nu,\mu)< A_H |\chi^{\mathrm{c}}(\nu)- \chi^{\mathrm{c}}(\mu)|.
$$
\end{lemma}

Proposition~\ref{pt.densityhomclass2} now follows exactly as the one of  
after noting that every $\eta \in \mathcal{M}_{\mathrm{erg}}^\alpha(f,H)$ is the limit of a sequence of periodic measures $(\nu_n)_n$.

To complete the proof of Proposition~\ref{pt.densityhomclass2} (and hence of Theorem~\ref{t.hyperboliccase})
 it remains to prove Lemma~\ref{l.apmethodgeral}.

\subsection{Proof of Lemma~\ref{l.apmethodgeral}}
\label{ss.proofoflemmaapmethodgeral}
To proof is similar to the one  of Proposition~\ref{p.mainprop}. The only difference is that we will interpolate between 
saddles with the same $\mathrm{s}$-index,
which is possible since they are homoclinically related.

  Fix $f \in \mathcal{R}_0(M)$, the homoclinic class $H$ of $f$, and the orbits
   $\mathcal{O}_{\ast}$,  $\ast \in \{1,2\}$, as in the lemma. For simplicity,  assume that the saddles have period one. Let
        \begin{equation}
        \label{eq.choicelambdahyp}
            \lambda \eqdef \max \left\{ \frac{\chi^{\mathrm{c}}_1}{2}, \frac{\chi^{\mathrm{c}}_2}{2},\log\lambda_\u\right\}<0,
            \quad \mbox{where  $\lambda_\u$  is as in Remark~\ref{r.lambdau1}.}
        \end{equation}
            Let $L$ and $a_0$ be the numbers given by Lemma~\ref{l.gan} for $\lambda$. We construct an orbit segment satisfying Conditions (1)-(3) of Lemma~\ref{l.gan}.
    As in the nonhyperbolic case, Condition (3) is satisfied, thus it is enough to check Conditions (1) and (2).

    Let $\delta>0$ and $0<\tau<1$. Recall the number $\alpha_{\inf}(H)$.
  By the continuity of $\varphi^\c$ in~\eqref{eq.derivadacentrala}, given 
        \begin{equation}
        \label{eq.varsigma}
            0< \varsigma< \min \left\{\frac{\delta}{4},\frac{|\chi^{\mathrm{c}}_1|}{8},\frac{|\chi^{\mathrm{c}}_2|}{8},|\alpha_{\min}|\tau\right\},
        \end{equation}
    there exists $\gamma>0$ such that 
        \begin{equation}
        \label{eq.continuidadedervivada}
            | \varphi_{\mathrm{c}}(x)-\varphi_{\mathrm{c}}(y)|<\varsigma \quad \mbox{for every } x,y \in M \,\, \mbox{such that  } d(x,y)<\gamma.
        \end{equation}

        Fix $\epsilon>0$ and
    \begin{equation}
    \label{eq.escolhaahyp}
        0< a < \min\left\{a_0, \frac{\epsilon}{L+1}, \frac{\gamma}{L+1} \right\}.
    \end{equation}

 For $\ast \in \{1,2\}$, consider families of disks $\mathfrak{D}(p_\ast)$ as in Remark~\ref{rd.familiadiscos}
whose disks are contained in the ball of radius $a/2$ centered at $p_\ast$.
The next claim follows from the invariance of these families  noting that saddles with the same $\mathrm{s}$-index are homoclinically related, recall (G2) in Theorem~\ref{t.residual}. It is the analogous to Lemma~\ref{l.sobrediscos} in the nonhyperbolic case. 

\begin{claim}
\label{c.l.constrdiscos}
There is $N$ with the following property.
    For every disk $D \in \mathfrak{D}(p_1)$ and every $\ell,r>0$ there are a disk $D' 
    =D'_{\ell, r}
    \subset D$ such that
    \begin{enumerate}
        \item $f^i(D')$ is contained in some disk of $\mathfrak{D}(p_1)$ for every $0 \leq i \leq \ell$;
        \item $f^{\ell+N+j}(D')$ is contained in some disk of $\mathfrak{D}(p_2)$ for every $0\leq j \leq r$; and
        \item $f^{\ell +N + r +N} (D')$ is a disk in $\mathfrak{D}(p_1)$.
    \end{enumerate}
\end{claim}

Given $D \in \mathfrak{D}(p_1)$ and $\ell,r>0$, let $D'$ be the disk given by Claim~\ref{c.l.constrdiscos}. Given $q \in D'$, define
    \begin{equation}
    \label{eq.defqlr}
        q_{\ell,r} \eqdef f^{N+r+N+\ell}(q) \in f^{\ell+N+r+N}(D') \in \mathfrak{D}(p_1). 
    \end{equation}
By construction, the orbit segment $[q_{\ell,r},\ell+N+r+N]$  satisfies Condition (1) in Lemma~\ref{l.gan}. 

As in the proof of Proposition~\ref{p.mainprop}, it is enough to verify Condition
(2) for the subbundle $E^{\mathrm{c}}$ of $E^{\mathrm{cs}}$. For that
consider the following remark.

\begin{remark}
\label{r.constrsegorbit}
    By the construction of the point $q_{\ell,r}$ and Claim~\ref{c.l.constrdiscos}, the orbit segment $[q_{\ell,r},\ell+N+r+N]$ satisfies: 
        \begin{itemize}
            \item $f^k(q_{\ell,r}) \in B_{a/2}(p_1)$ for $0\leq k \leq \ell$; and
            \item $f^k(q_{\ell,r}) \in B_{a/2}(p_2)$ for $\ell+N \leq k \leq \ell + N +r$.
        \end{itemize}
\end{remark}

In what follows, consider the number $L^\c \eqdef L^\c(H)$ as in~\eqref{e.Lc}.
\begin{claim}
\label{c.l.quasihyp}
   Let $\ell$ and $r$ sufficiently large such that 
        $$
        \frac{N}{\ell+N}  L^{\mathrm{c}} < \varsigma, \qquad \left|1-\frac{\ell}{\ell+N}\right| < \varsigma, \qquad \left|\frac{\chi^{\mathrm{c}}_2}{\ell+N+r}\right| < \varsigma.
        $$
        Then
$$
\frac{1}{k}\sum_{i=0}^{k-1} \varphi^{\mathrm{c}}(f^i(q_{\ell,r})) < \lambda,
            \qquad \mbox{for every $0 \leq k \leq \ell+N+r+N$.}
$$ 
\end{claim}
\begin{proof}
    There are four segments of the orbit of $q_{\ell,r}$ to consider, according to the value of $k$. 
  \smallskip  
    
 \noindent (i) $0\leq k \leq \ell$: 
 By Remark~\ref{r.constrsegorbit} and the choice of $a$ in~\eqref{eq.escolhaahyp}, we get 
 $$
 d(f^k(q_{\ell,r}),f^k(p_1))< a/2< \gamma.
 $$
    Hence, by~\eqref{eq.continuidadedervivada} and the choice of $\varsigma$ in~\eqref{eq.varsigma}, it follows that
        \begin{equation*}
        \frac{1}{k} \sum_{i=0}^{k-1}\varphi^{\mathrm{c}}(f^i(q_{\ell,r}))< \chi^{\mathrm{c}}_1+\varsigma< \frac{\chi^{\mathrm{c}}_1}{2}<\lambda.
    \end{equation*}

 \noindent (ii)  $ \ell + 1  \leq k \leq \ell + N$: 
 From the definition of $L^{\mathrm{c}}$ and $\chi^{\mathrm{c}}_1+\varsigma-L^{\mathrm{c}}(H(p))<0$, it follows
    \begin{equation*}
    \begin{split}
        \frac{1}{k}\sum_{i=0}^{ k-1} \varphi^{\mathrm{c}}(f^i(q_{\ell,r})) & < \frac{\ell}{ k } (\chi^{\mathrm{c}}_1 + \varsigma) + \frac{(k - \ell) }{k}  L^{\mathrm{c}} = L^{\mathrm{c}} + \frac{\ell}{k} ( \chi^{\mathrm{c}}_1 +\varsigma - L^{\mathrm{c}}) \\
        & < L^{\mathrm{c}} + \frac{\ell}{\ell +N}  ( \chi^{\mathrm{c}}_1 +\varsigma - L^{\mathrm{c}}) = \frac{\ell}{\ell + N}  (\chi^{\mathrm{c}}_1 + \varsigma) + \frac{N}{\ell+ N}  L^{\mathrm{c}}.
    \end{split}
    \end{equation*}
    By the choices of $\ell$ 
    and of $\varsigma$ (see ~\eqref{eq.varsigma}), 
    it follows     
    \begin{equation}
    \label{eq.escolhalhyp}
        \frac{1}{k}\sum_{i=0}^{ k-1} \varphi^{\mathrm{c}}(f^i(q_{\ell,r})) < \chi^{\mathrm{c}}_1 + 2\varsigma < \lambda,
        \quad \mbox{for  every$\ell+1 \leq k \leq \ell+N$.}
    \end{equation}

\smallskip

 \noindent (iii) $\ell+ N +1 \leq k \leq \ell + N + r$: The choice of $\varsigma$ in~\eqref{eq.varsigma}, implies that $\chi^{\mathrm{c}}_1+2\varsigma<0$ and $\chi^{\mathrm{c}}_2+\varsigma<0< \varsigma$. Therefore, 
    \begin{equation*}
    \begin{split}
        \frac{1}{k}\sum_{i=0}^{ k-1} \varphi^{\mathrm{c}}(f^i(q_{\ell,r})) & \leq \frac{1}{k} \sum_{i=0}^{ \ell + N -1} \varphi^{\mathrm{c}}(f^i(q_{\ell,r})) + \frac{1}{k} \sum_{i=\ell+N}^{k-1} \varphi^{\mathrm{c}}(f^i(q_{\ell,r})) \\
        & \leq \frac{\ell + N}{(\ell+N+r)}  (\chi^{\mathrm{c}}_1+2\varsigma) + \frac{1}{\ell + N + r}  (\chi^{\mathrm{c}}_2 + \varsigma) \\
       \footnotesize{ \mbox{choices $\ell,r$ and~\eqref{eq.varsigma}}} & < \chi_1^\c + 4\varsigma < \lambda. 
    \end{split}
    \end{equation*}

    \noindent (iv) $\ell+N+r+1 \leq k \leq \ell + r + 2N$: It follows that
    \begin{equation*}
    \label{eq.estimativafinalhyp}
    \begin{split}
        \frac{1}{k} \sum_{i=0}^{ k-1} \varphi^{\mathrm{c}}(f^i(q_{\ell,r})) & = \frac{1}{k}  \sum_{i=0}^{\ell-1} \varphi_{\mathrm{c}}(f^i(q_{\ell,r})) + \frac{1}{k}\sum_{i=\ell}^{\ell+N-1} \varphi_{\mathrm{c}}(f^i(q_{\ell,r}))\\
        & + \frac{1}{k}\sum_{i=\ell+N}^{\ell+N+r-1} \varphi_{\mathrm{c}}(f^i(q_{\ell,r})) +\frac{1}{k} \sum_{i=\ell+N+r}^{k} \varphi_{\mathrm{c}}(f^i(q_{\ell,r})) \\
        \footnotesize{\mbox{(Remark~\ref{r.constrsegorbit})}} & \leq \frac{\ell}{k}  (\chi^{\mathrm{c}}_1+\varsigma) + \frac{N}{k}  L^{\mathrm{c}} + \frac{r}{k}  (\chi^{\mathrm{c}}_2+\varsigma) + \frac{k-(\ell+N+r)}{k} L^{\mathrm{c}} \\ 
        & = L^{\mathrm{c}} + \frac{\ell}{k}  (\chi^{\mathrm{c}}_1 + \varsigma - L^{\mathrm{c}}(f)) + \frac{r}{k}  (\chi^{\mathrm{c}}_2 + \varsigma - L^{\mathrm{c}}) \\
        & \leq L^{\mathrm{c}} + \frac{\ell ( (\chi^{\mathrm{c}}_1 + \varsigma - L^{\mathrm{c}})) + r  (\chi^{\mathrm{c}}_2 + \varsigma - L^{\mathrm{c}})}{\ell + r +2N}  \\
        & \leq \frac{1}{\ell +r +2N}  \left( 
        \ell \chi^{\mathrm{c}}_1 + r \chi^{\mathrm{c}}_2+   2N  L^{\mathrm{c}}
        \right) + 2\varsigma \\
       \footnotesize{\mbox{($r$ large, choice of $\varsigma$ in~\eqref{eq.varsigma})}} & \leq \frac{\ell}{\ell+r}  \chi^{\mathrm{c}}_1 + \frac{r}{\ell+r}  \chi^{\mathrm{c}}_2 + 3\varsigma < \max\{\chi^{\mathrm{c}}_1,\chi^{\mathrm{c}}_2\} + 3\varsigma < \lambda.
    \end{split}
    \end{equation*}
    All possible cases have been considered, proving the
 claim.
\end{proof}

Since $[q_{\ell,r},\ell+N+r+N]$ satisfies conditions (1)--(3) of Lemma~\ref{l.gan}, there exists a periodic point $p_{\ell,r}$ with period $\ell+N+r+N$ such that
    \begin{equation}
    \label{eq.shadowed}
        d(f^i(p_{\ell,r}) , f^i(q_{\ell,r}))< La \qquad  \mbox{for every $i \in \{0, \dots, \ell+N+r+N\}$.}
    \end{equation}
The following claim estimates $\chi^{\mathrm{c}}(p_{\ell,r})$.

\begin{claim}
 \label{c.l.estimativaexpoenteapmethoddifeo}
     Let $\rho \in(\tau,1)$. For  sufficiently large $\ell$ and $r$, the following holds 
        $$
            \chi^{\mathrm{c}}(p_{\ell,r}) \in [\rho \chi^c_1 +(1-\rho)\chi^{\mathrm{c}}_2-\delta , \rho \chi^c_1 +(1-\rho)\chi^{\mathrm{c}}_2 + \delta].
        $$
\end{claim}
\begin{proof}
    For $0 \leq k \leq \ell$, using~\eqref{eq.shadowed}, the choice of $a$  in~\eqref{eq.escolhaahyp}, and Remark~\ref{r.constrsegorbit}
    it follows that
        \begin{equation}
        \begin{split}
        \label{eq.primeirapartehyp}
            d(f^k(p_{\ell,r}),p_1) & < d(f^k(p_{\ell,r}), f^k(q_{\ell,r})) +  d(f^k(q_{\ell,r}),p_1)
            < (L+1)a < \gamma.
        \end{split}
        \end{equation}
    Analogously, for $\ell + N  \leq k \leq \ell +N + r$, it holds that $d(f^k(p_{\ell,r}),p_2) < \gamma$.
    Therefore, using~\eqref{eq.continuidadedervivada} we get  
        \begin{equation}
        \label{eq.intexp}
        \begin{split}
            & |\varphi^{\mathrm{c}}(f^k(q_{\ell,r}))-\varphi^{\mathrm{c}}(p_1)|< \varsigma \quad \mbox{for } 0 \leq k \leq \ell \quad \mbox{and}\\
            & |\varphi^{\mathrm{c}}(f^k(q_{\ell,r}))-\varphi^{\mathrm{c}}(p_2)|< \varsigma \quad \mbox{for } \ell + N \leq k \leq \ell+ N + r.
        \end{split}
        \end{equation}
    By \eqref{eq.intexp}, the definition of $L^{\mathrm{c}}$, and  $\pi(p_{\ell,r})=\ell+r+2N$, we get 
    \begin{equation}
    \begin{split}
    \label{eq.estimativaexphyp1}
        \chi^{\mathrm{c}}(p_{\ell,r})& = \frac{1}{\ell+r+2N} \sum_{k=0}^{\ell+r+2N-1}\varphi^{\mathrm{c}}(f^k(p_{\ell,r}))\\
        & \leq
        \frac{\ell(\chi^{\mathrm{c}}_1+\varsigma) + r(\chi^{\mathrm{c}}_2+\varsigma)+ 2NL^{\mathrm{c}}}{\ell+r+2N} 
     \leq 
        \frac{\ell  \chi^{\mathrm{c}}_1 + r  \chi^{\mathrm{c}}_2 + 2NL^{\mathrm{c}}}{\ell+r+2N} 
        + 2\varsigma.
    \end{split}
    \end{equation}
    A similar calculation gives
    \begin{equation}
    \label{eq.estimativaexphyp2}
        \chi^{\mathrm{c}}(p_{\ell,r}) = \frac{1}{\ell+r+2N} \sum_{k=0}^{\ell+r+2N-1}\varphi^{\mathrm{c}}(f^k(p_{\ell,r}))
        \geq 
        \frac{\ell  \chi^{\mathrm{c}}_1 + r  \chi^{\mathrm{c}}_2  - 2 N L^{\mathrm{c}}}{\ell+r+2N} .
    \end{equation}
    Fix $\ell$ and $r$ sufficiently large such that 
   
       \begin{fact}
    \label{f.af.boundtole1e2}
    For every sufficiently large $\ell$ and $r$ such that
        $$
        - \frac{\varsigma}{3 |\alpha_{\inf}(H)|} < \frac{\ell}{\ell+r+2N} -\rho <  \frac{\varsigma}{3|\alpha_{\inf}(H)|}
        \qquad  \mbox{and} \qquad
        \frac{2N}{\ell+r+2N}   <  \frac{\varsigma}{3L^{\mathrm{c}}},
$$
it holds 
        \begin{equation}
        \begin{split}
            &\rho \chi^{\mathrm{c}}_1 - \frac{\varsigma}{3} < \frac{\ell \chi^{\mathrm{c}}_1}{\ell+r+2N} < \rho \chi^{\mathrm{c}}_1 + \frac{\varsigma}{3},\\
       &      (1-\rho) \chi^{\mathrm{c}}_2 - \frac{2\varsigma}{3} < \frac{\ell  \chi^{\mathrm{c}}_2}{\ell+r+2N} < (1-\rho)\chi^{\mathrm{c}}_2 + \frac{2\varsigma}{3}.     
       \end{split}
       \end{equation}
    \end{fact}
    
        \begin{proof}
    We get the upper bound. The calculations to obtain the lower bound are analogous.     
    Observe that taking $\ell$ and $r$ as in the hypotheses it follows
   $$
        1= \frac{\ell+r+2N}{\ell+r+2N} 
        < \rho + \frac{\varsigma}{3|\alpha_{\inf}|(H)} + \frac{r}{\ell+N+r+N} + \frac{\varsigma}{3L^{\mathrm{c}}}.
  $$
    Which implies that
    \begin{equation*}
        \frac{r}{\ell+r+2N} > 1-\rho - \frac{\varsigma}{3|\alpha_{\inf}(H) |} -\frac{\varsigma}{3L^{\mathrm{c}}}.
    \end{equation*}
    Since $\alpha_{\inf}(H) <\chi^{\mathrm{c}}_2<0 < L^{\mathrm{c}}$, it follows that
    \begin{equation*}
        \frac{r \chi^{\mathrm{c}}_2}{\ell+r+2N} < (1-\rho)\chi^{\mathrm{c}}_2 + \frac{\varsigma \chi^{\mathrm{c}}_2}{3\alpha_{\inf}(H)} + \frac{\varsigma|\chi^{\mathrm{c}}_2|}{3L^{\mathrm{c}}} < (1-\rho)\chi^{\mathrm{c}}_2 + \frac{2\varsigma}{3},
    \end{equation*}
    ending the proof of the fact.
    \end{proof}

    Using~\eqref{eq.estimativaexphyp1} and Fact~\ref{f.af.boundtole1e2}, we get 
       \begin{equation*}
    \begin{split}
        \chi^{\mathrm{c}}(p_{\ell,r}) & < \rho \chi^{\mathrm{c}}_1 + \frac{\varsigma}{3}+ (1-\rho)\chi^{\mathrm{c}}_2+ \frac{2 \varsigma}{3} + \frac{\varsigma}{3} + 2\varsigma \\
        &= \rho \chi^{\mathrm{c}}_1 +(1-\rho)\chi^{\mathrm{c}}_2 + \frac{10 \varsigma}{3} < \rho \chi^{\mathrm{c}}_1 +(1-\rho)\chi^{\mathrm{c}}_2 + 4\varsigma.
    \end{split}
    \end{equation*}
    Analogously, it holds
    $$
        \chi^{\mathrm{c}}(p_{\ell,r})> \rho \chi^{\mathrm{c}}_1 + (1-\rho)\chi^{\mathrm{c}}_2 - 4\varsigma.
    $$
    Since $\varsigma< \frac{\delta}{4}$,
    $$
        \chi^{\mathrm{c}}(p_{\ell,r}) \in [\rho \chi^c_1 +(1-\rho)\chi^{\mathrm{c}}_2-\delta , \rho \chi^c_1 +(1-\rho)\chi^{\mathrm{c}}_2 + \delta].
    $$
Finishing the proof of the claim.
\end{proof}

For $\ell$ and $r$ as in Claim~\ref{c.l.estimativaexpoenteapmethoddifeo}, consider the periodic orbit $\widetilde{\mathcal{O}}\eqdef \mathcal{O}(p_{\ell,r})$. It remains to check the following:

\begin{fact}
The orbit
 $\widetilde{\mathcal{O}}$ is an $(\epsilon,\rho-\tau)$-good approximation of $\mathcal{O}_1$. 
 \end{fact}
 
 \begin{proof}
 Note  that $\rho \in (\tau,1)$.
Proceeding as in~\eqref{eq.primeirapartehyp}, it follows that 
        $
            d(f^k(p_{\ell,r}),p_1) <(L+1)a < \epsilon
        $
        for all $0 \leq k \leq \ell$.
Thus, by Fact~\ref{f.af.boundtole1e2}
and the choice of $\varsigma$ in~\eqref{eq.varsigma}, we get 
    \begin{equation*}
    \begin{split}
        \frac{\#(\widetilde{\mathcal{O}}_{\epsilon}^{\mathrm{g}})(\mathcal{O}_1)}{\#(\widetilde{\mathcal{O}})} \geq \frac{\ell}{\ell+r+ 2N} > \rho - \frac{\varsigma}{3|\alpha_{\inf}|}> \rho -\tau,     
    \end{split}
    \end{equation*}
proving  the fact,
\end{proof}

The proof of  the lemma is now complete.
\qed

\section{Proof of Theorem~\ref{t.densitygeraldifeo}}
\label{s.proofthmB}
In this section, we derive Theorem~\ref{t.densitygeraldifeo} from our previous 
constructions.

\subsection{Preliminaries}
\begin{lemma}
\label{lp.flipflopmb}
    Let $f \in \mathbf{MB}^1(M)$, $B^+$ be an unstable blender-horseshoe of $f$, and $p$ a saddle of $f$ with $\chi^\c(p)<0$. Then $B^+$ and $p$ form a split flip-flop configuration.
\end{lemma}

There is a similar statement for stable blender-horseshoes and saddles with $\chi^\c(p)>0$.

\begin{proof}[Proof of Lemma~\ref{lp.flipflopmb}]
  Consider $B^+$ and $p$ as in the statement. Assume, for simplicity that, $p$ is a fixed point.
  Due to minimality of the strong unstable foliation and the fact
  that the family of disks in between  $\mathfrak{D}_{\mathrm{bet}}(B^+)$ is open, 
  there is a number $R$ such that  the set $\mathcal{F}^\u_R(x)$ contains a disk in 
  $\mathfrak{D}_{\mathrm{bet}}(B^+)$ for every $x\in M$.
  Note that since $\chi^\c(p)<0$, it holds $W^{\mathrm{u}} (p)= \mathcal{F}^\u (p)$.
  
  We consider a family of disks $\mathfrak{D}(p)$ as in Remark~\ref{rd.familiadiscos}.
  The next claim implies that $B^+$ and $p$ forms a split flip-flop configuration.

    \begin{claim}
    \label{cl.mbsplitflipflop}
        There exists $N_{\delta,p}>0$ such that for every $n \geq N_{\delta,p}$ it holds that
        \begin{itemize}
            \item $f^n(\mathcal{W}^{\mathrm{u}}_{\delta}(p))$contains a disk in $\mathfrak{D}_{\mathrm{bet}}(B^+)$; and
            \item for every $D \in \mathfrak{D}_{\mathrm{bet}}(B^+)$, the set $f^n(D)$ contains a disk $D' \in \mathfrak{D}(p)$.
        \end{itemize}
    \end{claim}
  
  \begin{proof}
  For the first assertion, note that the $f$ is uniformly expanding along the unstable direction. Hence there is
  $N_0$ such that $f^{n} (\mathcal{W}^{\mathrm{u}}_{\delta}(p)) \supset\mathcal{W}^{\mathrm{u}}_{R}(p)$
  for every $n \ge N_0$.
  The assertion now follows from the choice of $R$.
  
  The minimality of the strong stable foliation provides $R'>0$ such that
  $W^{\mathrm{s}}_{R'}(p)$ transversely intersects every disk in   $\mathfrak{D}_{\mathrm{bet}}(B^+)$. The second assertion in the claim follows now from the $\lambda$-lemma.
This proves the claim.
  \end{proof}
  The proof of the lemma is now complete.
  \end{proof}
 
\subsection{End of the proof of Theorem~\ref{t.densitygeraldifeo}}
We consider 
first nonhyperbolic measures ($\alpha=0$). We prove that  for every $f\in \mathbf{MB}^1(M)$ 
the whole $M$
satisfies hypotheses in Theorem~\ref{t.densityisolatedset} and hence we are done. 
The occurrence of the split flip-flop configuration follows
from Lemma~\ref{lp.flipflopmb}.
For the second condition,
 note that by \cite{DiaGelSan:20,YanZha:20} every ergodic  nonhyperbolic measure is
 (simultaneously)
  a limit of
hyperbolic ones with central positive and negative exponents. As in our partially hyperbolic setting these measures are approached
by periodic ones \cite{Cro:11,Gel:16}, the second condition holds. This completes the case $\alpha=0$.

For the hyperbolic values $\alpha\ne 0$, we use a consequence of the minimality 
of the strong foliations, which is a strong version of Condition~\ref{G2} in Theorem~\ref{t.residual}.

\begin{lemma}
\label{lp.homreldifeo}
    Let $f \in \mathbf{PH}_{\mathrm{c}=1}^1(M)$ such that both strong stable and unstable foliations are minimal. Then every pair of saddles $p$ and $q$ of $f$ with the same $\mathrm{s}$-index are homoclinically related.
\end{lemma}

This means we are in a  setting  similar to the one in Theorem~\ref{t.hyperboliccase} and can proceed similarly.
\qed

\bibliographystyle{plain}

\end{document}